\documentclass[a4paper,12pt]{article}
\usepackage{bm}
\usepackage{amsmath}
\usepackage{amsthm}
\usepackage{wasysym}
\usepackage{amssymb}
\usepackage{amsfonts}
\usepackage{graphicx}
\usepackage[english]{babel}
\usepackage[utf8]{inputenc}
\usepackage[T1]{fontenc}
\usepackage{mathrsfs}
\usepackage{enumerate}
\usepackage{version}
\usepackage{calc}
\usepackage{subfigure}
\usepackage{authblk}
\usepackage{comment}
 \usepackage{pstricks,pst-math,pst-xkey}
 \usepackage{float}
 \usepackage{indentfirst}
\usepackage[pagewise]{lineno} 
\newtheorem{definition}{Definition}[section]
\newtheorem{theorem}[definition]{Theorem}
\newtheorem{lemma}[definition]{Lemma}
\newtheorem{corollary}[definition]{Corollary}
\newtheorem{proposition}[definition]{Proposition}
\newtheorem{remark}[definition]{Remark}

\numberwithin{equation}{section}

\newcommand{\N}{{\mathbb N}}

\newcommand{\R}{{\mathbb R}}

\newcommand{\ve}{{\varepsilon}}

\def\b{\beta}

\def\ve{\varepsilon}

\def\f{\varphi}

\def\d{\partial}

\def\ti{\tilde}
\def\tu{\tilde{u}}
\def\tuv{\tilde{u}_{\varepsilon}}

\def\w{\omega}

\def\wN{\tilde{w}_N^\ve}
\def\vN{\tilde{v}_N^\ve}

\def\N{\mathbb{N}}

\def\F{\mathcal{F}}
\def\T{\mathbb{T}_L}

\def\P{\mathbb{P}}

\def\f{\varphi}

\def\R{\mathbb{R}}
\def\E{\mathbb{E}}
\def\N{\mathbb{N}}
\def\P{\mathbb{P}}

\begin{document}

\title{Long Time Behavior of Stochastic Thin Film Equation}

\author[1]{Oleksiy Kapustyan \thanks{kapustyanav@gmail.com}}

\author[3] {Olha Martynyuk
\thanks{o.martynyuk@chnu.edu.ua}}

\author[2]{Oleksandr Misiats\thanks{omisiats@vcu.edu}}

\author[1]{Oleksandr Stanzhytskyi \thanks{ostanzh@gmail.com}}

\affil[1]{Department of Mathematics,
Taras Shevchenko National University of Kyiv, Ukraine}

\affil[2]{Department of Mathematics, Virginia Commonwealth University,
Richmond, VA, 23284, USA}

\affil[3]{Department of Mathematics, Yuriy Fedkovych Chernivtsi National University, Chernivtsi, Ukraine}



\maketitle

\begin{abstract}
In this paper we consider a stochastic thin-film equation with a one dimensional Gaussian Stratonovych noise. We establish the existence of non-negative global weak martingale solution, and study its long time asymptotic properties.  In particular, we show the solution almost surely converges to the average value of the initial condition. Furthermore, using the regularized equations and adapted entropy functionals, we establish the exponential asymptotic decay of the solution in the uniform norm. 
\end{abstract}

\section{Introduction.} We consider the  stochastic thin-film  equation with quadratic mobility
\begin{equation}\label{1.1}
du = (-\partial_x(u^{2}u_{xxx}))dt + \partial_x (u \circ d\b)
\end{equation}
for $(t,x) \in Q_T:=[0,T) \times \mathbb{T}_L$. Here $T>0, L>0$, and $\mathbb{T}_L$ is the torus on the interval $[0,L]$, with periodic boundary conditions
\[
\partial_x^i u(\cdot, 0) = \partial_x^i u(\cdot, L), \ i = 0,1,2,3,
\]
and non-negative initial condition $u(0,x) = u_0(x)$. The term $\partial_x (u \circ d\b)$ is a stochastic perturbation in Stratonovich form, and the process $\beta$ is a Wiener process on a complete filtered probability space $(\Omega, \mathcal{F}, \{\mathcal{F}_t\}, \mathcal{P}), t \in [0,T]$, with a complete and right-continuous filtration $(\mathcal{F}_t)_{t \in [0,T]}$.
Similarly to \cite{Gess1} and \cite{KapMisSta23}, the equation \eqref{1.1} may be re-written in the Ito's form 
\begin{equation*}\label{1.2}
du = \left(\partial_x(-u^{2}u_{xxx}) + \frac{1}{2}  \partial^2_x  u \right)dt +
\partial_x u \, d \beta(t).
\end{equation*}

The equations of type  \eqref{1.1}  model the motion of liquid droplets of thickness $u$, spreading over the solid surface. It can be derived from the lubrication theory under the so-called ``thin-film assumption'', that is, that the dimensions in the horizontal directions are significantly larger than in the vertical (normal) one. Under this assumption, the surface tension plays the key role in the dynamics of the droplet. The regions where  $u>0$ are called the {\it wetted regions}. The equation is parabolic inside these regions, and degenerate on the boundary of these regions. This way one may think of  \eqref{1.1} as of a  fourth-order nonlinear free boundary problem inside a wetted region, which evolves in time  with finite propagation speed \cite{Bernis}. 

The mathematical analysis of equations of type \eqref{1.1} is particularly tricky due to the lack of comparison (maximum) principle, which is one of the primary analytical tools in classic parabolic theory. One of the pioneering works in this direction was the paper  \cite{Bern} by Bernis and Friedman. In it, the authors established the existence of a non-negative generalized weak solution using the energy-entropy method. This solution was obtained as a limit of the solutions of the regularized problems. It is worth noting that in this work, the authors introduce the notion of a weak solution via an integral identity inside wetted regions. This notion is somewhat ``weaker'' than the subsequent definitions, which involve the integral identities over the entire region $\mathbb{T}$.  In \cite{Bern} the authors also describe the behavior of the support of the solution. 

 The work \cite{DalPas} addressed existence of more regular (strong or entropy) solutions.  In this work the authors also studied the asymptotic in $t$ behavior of the solution. In particular, it was shown that under certain assumptions on the initial conditions,
 \begin{equation}\label{1.3}
 u(t,x) \to \frac{1}{L} \int_0^L u_0(x) \, dx, \ t \to \infty.
 \end{equation}
 A similar result was obtained in \cite{BertPugh} using a somewhat different approach using entropy estimates. Furthermore, this work established that the rate of the aforementioned convergence is exponential.  The work \cite{Tud} makes use of the energy estimates for the regularized problem to obtain a stronger result, that is, the exponential decay of the energy
 \[
 J[u]:= \frac{1}{2} \int_0^L u_x^2(t,x) \, dx.
 \]
The work \cite{KapTar} established the existence of a global trajectory attractor for the generalized thin film equation with a nonlinear dissipative term, which describes the relation between nonlinear absorption and spatial injection. 

In this paper we consider the thin film equation with stochastic perturbation. It is worth mentioning that one needs to be very careful when dealing with the effect of noise on nonlocal and/or ill-posed problems. In some cases, e.g. \cite{MisStaTop}, the presence of a even small stochastic perturbation leads to a finite time blowup while an unperturbed equation has global solution. In others \cite{Sta}, the effect of the random perturbation is exactly opposite - it may lead to the existence of a global solution while the corresponding deterministic equation has a finite time blowup. The long time behavior of stochastically perturbed evolution equations is typically described via the existence and properties of invariant measures, see, e.g. \cite{MisStaYip1}, \cite{MisStaYip2}, \cite{MisStaYip3}, \cite{MisStaSta}, \cite{MisStaHie}, \cite{MisStaKap}, \cite{MisStaCla}.

The stochastic version of thin-film equation was first introduced in \cite{17}, it described  modeling the enhanced spreading of droplets.  In the subsequent work \cite{33} the authors additionally take the interface potential between fluid and substrate into account. This prevents the solution $u$ from becoming negative, and allows to describe the coarsening and de-wetting phenomena. 

The first rigorous construction of a non-negative martingale solution of the stochastic thin film equation with Ito noise and additional interface potential was obtained in \cite{19}. This result was derived by constructing a spatially discrete approximation of the solution.  In \cite{13}, the author considered a more general case of the main operator in the form $-\partial_x(u^{n}u_{xxx})$ (referred as a more general {\it{mobility}}), and established the conditions for the existence of a global strong solution for this problem. 

The work \cite{Gess1}  established the existence of a nonnegative matringale solution for \eqref{1.1} with quadratic mobility.  The main tool in \cite{Gess1} was Trotter-Kato type decomposition, which allowed to separate the dynamics into deterministic and stochastic parts, which are eventually coupled. The subsequent work \cite{Gess2} established the similar result for
\[
du  = - \partial_x(u^n \partial_x^3 u)\, dt + \partial _x (u^{\frac{n}{2}} \circ  dW)
\]
for $n \in [\frac{8}{3}, 4)$.  To this end, the authors started with regularizing the problem by replacing the mobility $u^n$ with $u^n \to (u^2+\ve^2)^{n/4}, \ \ve>0$, which makes the problem non-degenerate. Then, by means of Galiorkin approximations, the authors  establish the existence of a weak solution for the regularized equation, and  pass to the limit $\ve \to 0$.  

In \cite{KapMisSta23} we established the existence of a non-negative martingale solution of the nonlinear stochastic thin-film  equation with nonlinear deterministic and stochastic drift coefficients
\begin{equation*}
du = (-\partial_x(u^{2}u_{xxx})  + l(u))dt + \partial_x (u \circ dW)+ f(u)dW_1(t),
\end{equation*}
where $W$ and $W_1$ are independent $Q$-Wiener processes. 
The above equation has both nonlinear drift $l(u)$, as well as the nonlinear stochastic Ito perturbation $f(u) d W_1(t)$. Due to the presence of these two terms, the equation is no longer in divergence form, which makes its analysis significantly different from the one in \cite{Gess1} and \cite{Gess2}. 

To the best of our knowledge, the question of asymptotic long-time behavior of solutions of stochastic thin film equation has not been addressed in literature. Let us outline the strategy of obtaining the main result on the long time behavior, the convergence \eqref{1.3}. \\

{\bf Step 1.} (Regularization) For any $\ve>0$ consider
\begin{equation}\label{1.4}
du_\ve = (-\partial_x(f_\ve(u_\ve) \d_x^3 u_\ve ) + \frac{1}{2} \d_x^2 u_\ve)dt  + \partial_x u_\ve d\b,
\end{equation}
where 
\[
f_\ve(s) = \frac{s^4}{\ve+s^2}
\]
and the initial condition
\begin{equation*}
{u_0}_\ve = u_0 + \delta(\ve), \ \delta(\ve) \to 0 \text{ as } \ve \to 0. 
\end{equation*}

\begin{definition}\label{Def1.1}
A weak martingale solution of \eqref{1.4} with $\tilde{\mathcal{F}}_0$ - measurable initial data $u_{0,\ve} \in L^2(\Omega, H^1([0,T], \R_0^+))$ is a quadruple
\[
\{ (\tilde{\Omega}, \tilde{\mathcal{F}}, \tilde{\mathcal{F}}_t, \tilde{\mathcal{P}}), \tilde{\b}, \tilde{u}_{0 \ve}, \tilde{u}_\ve\},
\]
such that $(\tilde{\Omega}, \tilde{\mathcal{F}}, \tilde{\mathcal{F}}_t, \tilde{\mathcal{P}})$ is a filtered probability space, $\tilde{u}_{0 \ve}$ is $\tilde{\mathcal{F}}_0$ measurable and has the same distribution as $u_{0 \ve}$, $\tilde{\beta}$ is a real-valued $\tilde{\mathcal{F}}_t$-Wiener process, and $\tilde{u}_{\ve}$ is an $\tilde{\mathcal{F}}_t$-adapted continuous $\mathbb{H}^1(\T)$ - valued process, such that:\\
1) $\tilde{\E} \sup_{t \in [0,T]} \|\tilde{u}_\ve(t)\|_{H^1(\T)}^2 < \infty$;\\
2) For almost all $(\omega, t) \in \tilde{\Omega} \times [0,T]$, the weak third order derivative $\d_x^3 \tilde{u}_\ve$ exists and satisfies 
\[
\tilde{\E} \|f_{\ve}(\tilde{u}_\ve)  \d_x^3 \tilde{u}_\ve\|^2_{L^2(Q_T)} < \infty;
\]
3) For any $\f \in C^\infty(\T)$, $d \tilde{P}$ - almost surely, we have

\begin{multline}\label{1.5}
     (\ti{u}_\ve(t,\cdot), \f)_2 -  (u_{0\ve}, \f)_2 = \int_0^t  \int_0^{L} f_\ve(\ti{u}_\ve) (\d_x^3 \ti{u}_\ve(s,\cdot)) \d_x \f dx ds  
    - \frac{1}{2}  \int_0^t ( \d_x( \ti{u}_\ve(s, \cdot)), \d_x \f)_2 ds \\
    - \int_0^t \left(\int_0^{L} \ti{u}_\ve \d_x \f dx\right) d \tilde{\beta}(s).
\end{multline}
\end{definition}

\begin{definition}\label{Def1.1*}
A weak martingale solution of \eqref{1.1} is $\tilde{u}(t)$ which satisfies 1) and 2) of Definition \ref{Def1.1}, while instead of \eqref{1.5} we have
\begin{multline*}
(\ti{u}(t,\cdot), \f)_2 -  (u_{0}, \f)_2 = \int_0^t  \int_{\{\tilde{u}(s)>0\}} f_\ve(\ti{u}) (\d_x^3 \ti{u}(s,\cdot)) \d_x \f dx ds  
    - \frac{1}{2}  \int_0^t ( \d_x( \ti{u}(s, \cdot)), \d_x \f)_2 ds \\
    - \int_0^t \left(\int_0^{L} \ti{u} \d_x \f dx\right) d \tilde{\beta}(s)
\end{multline*}
\end{definition}
 
for all $t \in [0,T]$.
 
{\bf Step 2.} Next we show that the solution $\tilde{u}_\ve(t)$ can be extended globally to $[0,\infty)$ (possibly on a different probability space).\\

{\bf Step 3.} We apply Ito's formula for the energy functional
\[
J_\ve[\tuv] := \frac{1}{2} \int_0^{L} (\d_x \tuv(t,x))^2 \, dx
\] 
to obtain an a priori bound
\begin{equation}\label{1.6}
J_\ve[\tuv(t, \cdot)] \leq J_\ve[\tu_{0 \ve}] e^{K_\ve^2 t} + K_\ve C(\ve) \int_0^t  \int_0^L (\d_x^2 \tuv)^2 \, dx ds
\end{equation}
for some positive constants $K_\ve$ and $C(\ve)$ almost surely.

{\bf Step 4.} For any $T>0$ we pass to limit as $\ve \to 0$ on $[0,T]$ in \eqref{1.4}, hence establishing the existence of a global solution for \eqref{1.1} in the sense of the Definition \ref{Def1.1*}. Furthermore, we show $\d_x \tuv \rightharpoonup \d_x \tu, \ \ve \to 0$, and $\sup_{t \in [0,T]} \int_0^L (\d_x^2 \tuv)^2 \, dx$ is bounded a.s. uniformly in $\ve$.

{\bf Step 5.} We conclude with passage to the limit in \eqref{1.6} as $\ve \to 0$ and obtain the desired result.

The paper is structured as follows: in Section \ref{Sec2} we introduce the notation, list preliminary results, formulate the main results. In Sections \ref{Sec3} through \ref{Sec6} we derive the local-in-time solution of the regularized problem, and study its properties. In Section \ref{Sec7} we show that the local solution of the regularized problem, obtained in the previous sections, can be extended globally for all $t \geq 0$.  Finally, in Section \ref{Sec8} we pass to the limit in the regularized problem, and therefore complete the proof of the main result.

\section{Preliminaries and main results.}\label{Sec2}
Throughout the paper, we will be using the following notation. 
For $u,v \in L^2(\mathbb{T}_L)$, let
\[
(u,v)_2 := \int_0^L u(x) v(x) dx := \int_L u(x) v(x) dx, \text{ and } \|u\|_2:= \sqrt{(u,u)_2},
\]
where $\int_L (\cdot) \, dx$ will be denoting  $\int_0^L (\cdot) \, dx$.
Next, for $Q \in \R^d$ with $\partial Q \in C^{\infty}$, for $s \in [0, \infty)$, $p \in [1, \infty)$, let $W^{s,p}(Q)$ be the regular Sobolev space for $s \in \N$, and Sobolev-Slobodeckij space for non-integer $s$. For $p=2$ we will denote $W^{s,p}(Q) = H^s(Q) = H^s$. If $X$ is a Banach space, the space $C^{k+\alpha}(Q;X)$ is the space of $k$ times differentiable functions $Q \to X$, whose $k$-th derivatives are Holder-continous with exponent $\alpha \in (0,1)$ on compact subsets of $Q$. We also denote $C^{k-}(Q;X)$ to be the space of $k-1$-times differentiable functions, whose $k-1$-st derivatives are Lipschitz continuous. The space $BC^{0}(Q,X)$ is the set of bounded continuous functions.  We also denote $H^1_w(\mathbb{T}_L)$ to be the space of  $H^1(\mathbb{T}_L)$ functions endowed with the weak topology, induced by $\| \cdot \|_{1,2}$. 
Finally, for  $u: Q_T \to \R$ we denote 
\[
P_T := \{(t,x) \in \bar{Q}_T, u(t,x)>0\}.
\]
Our main result is an immediate consequence of the theorem below:

\begin{theorem}\label{Thm2.1}
Asuume $u_0 \in H^1$ is such that $u_0 \not\equiv 0$, $u_0 \geq 0$ and 
\begin{equation}\label{2.1}
\int_{L} |\ln u_0| \, dx < \infty.
\end{equation}
Then the equation \eqref{1.1} has a weak martingale solution for $t \geq 0$ such that
\begin{equation}\label{2.2}
J[u(t)] \leq J[u_0] e^{-Ct},  \ t \geq 0 \text{ a.s.},
\end{equation}
where $C>0$ depends only on $J[u_0]$. 
\end{theorem}
Indeed, due to the conservation of mass property $$\overline{u_0}:=\frac{1}{L}\int_L u_0 \, dx = \frac{1}{L} \int_L u(t,x) \, dx =: \overline{ u(t) }, \ t \geq 0,$$
the Sobolev embedding and Poincare inequality imply
\begin{equation*}\label{2.3}
\|u(t,x) - \overline{u_0}\|_{\infty}^2 =  \|u(t,x) - \overline{u(t)} \|_{\infty}^2 \leq C\|u(t, \cdot) -   \overline{ u(t)}\|_{H^1} \leq C_1 \int_L u_x^2(t,x) \, dx \to 0
\end{equation*}
with probability 1. This way, for large enough $t$ the solution of \eqref{1.1} becomes nonnegative and converges to the average value of $u_0$.
\begin{remark}
Theorem \ref{Thm2.1} can be extended to the case of random  initial data $u_0$, which satisfies $u_0 \in L^q(\Omega, \mathcal{F}_0, \mathbb{P}, H^1(\T))$ for $q \geq 2$ large enough (see Proposition \ref{Prop3.2} and formula \eqref{3.16} for the precise condition on $q$) and $u_0 \geq 0$  $\mathbb{P}$ - almost surely. Furthermore, if
\begin{equation*}
\E \int_{L} |\ln u_0| \, dx < \infty,
\end{equation*}
the estimate \eqref{2.2} holds as well. 
\end{remark}

\section{The analysis of regularized problem.}\label{Sec3} For any $\ve>0$ consider the regularized equation

\begin{equation}\label{3.1}
du_\ve = \left(-\partial_x\left(f_\ve(u_\ve) \d_x^3 u_\ve \right) + \frac{1}{2} \d_x^2 u_\ve\right)dt  + \partial_x u_\ve d\b,
\end{equation}
along with the regularized initial condition
\begin{equation}\label{3.2}
u_{0 \ve} = u_0 +\delta(\ve) = u_0 +\ve^\theta.
\end{equation}
Here the random initial condition $u_{0 \ve}$ satisfies $u_0(\omega) \in H^1$, $u_0 \geq 0$ a.s.,  and $\E \|u_0\|^q_{1,2} < \infty$ for sufficiently large $q \geq 2$ (specified in Proposition \ref{Prop3.2}), as well as $\theta \in (0, \frac{2}{5})$. The function $f_\ve(s) = \frac{s^4}{\ve+s^2}$ clearly satisfies $f_\ve(s) \leq s^2$, as well as $f_\ve(s)$ is non-decreasing. Furthermore, as follows from \cite{BertPugh}, $f_\ve(s) \to s^2$ and $f_\ve^{'}(s) \to 2s$ uniformly on $[0,M]$ as $\ve \to 0$, for any $M>0$. 
\begin{theorem}\label{Thm3.1}
For any $\ve>0$ and $T>0$ the problem \eqref{3.1} has a weak strictly positive martingale solution on $[0,T]$ in the sense of Definition \ref{Def1.1}. Furthermore, this solution has the mass preservation property
\[
\int_{L} u_\ve (t,x) \, dx = \int_{L} u_{0 \ve}(x) \, dx.
\] 
\end{theorem}
\begin{proof}
The strategy of the proof is to apply Trotter-Kato type decomposition of the dynamics into deterministic and stochastic ones. This method was used by a number of authors, e.g. \cite{Gess1, ManZau}, in particular, to establish the existence of for SPDEs with localy Lipschitz coefficients. For the convenience of the reader, we split the proof into the steps below.\\
{\bf Step 1: Deterministic dynamics.} Consider the following deterministic equation:
\begin{equation}\label{3.4}
\begin{cases}
\d_t v_\ve = -\partial_x(f_\ve(v_\ve) \d_x^3 v_\ve ); \\
v_{\ve}(0) :=v_{\ve 0}.
\end{cases}
\end{equation}
It follows from \cite{Bern, DalPas} that \eqref{3.4} has a unique positive smooth solution $v_\ve$ satisfying
\[
\| \partial_x v_\ve\|_2^2 + 2 \int_0^t \int_{L} f_\ve(v_\ve) (\d_x^3 v_\ve)^2 \, dx ds = \|\d_x v_{0 \ve}\|_2^2, \ t \in [0,T].
\]
Hence, for any $p \geq 2$,
\begin{multline}\label{3.5}
\| \d_x v_\ve\|_2^p + 2 \int_0^t \|\d_x v_\ve(s, \cdot)\|_2^{p-2} \int_{L} f_\ve(v_\ve)(\d_x^3 v_\ve)^2 \, dx ds \\
\leq \sup_{s \in [0,t]} \| \d_x v_\ve\|_2^{p-2} \left(\| \d_x v_\ve\|_2^2 + 2 \int_0^t \int_{L} f_\ve(v_\ve)(\d_x^3 v_\ve)^2 \, dx ds\right) \leq  \|\d_x v_{0 \ve}\|_2^{p-2}  \|\d_x v_{0 \ve}\|_2^{2} =  \|\d_x v_{0 \ve}\|_2^{p}.
\end{multline}
{\bf Step 2: Stochastic dynamics.}  Let $(\Omega, {\F}, ({\F}_{t})_{t \in [0,T]}, \P)$ be a filtered probability space, where $({\F}_{t})_{t \in [0,T]}$ is a complete and right-continuous filtration, and $\beta(t)$ is standard real-valued ${\F}_{t}$ -Wiener process. Consider
\begin{equation}\label{3.6}
\begin{cases}
d w  = \frac{1}{2} \d_x^2 w \, dt + \d_x w \, d \beta(t) \ \text{ on } [0,T];\\
w(0) = w_0 \in L^p(\Omega, \F_0, \P; H^1(\T)).
\end{cases}
\end{equation}
\begin{definition}
An ${\F}_t$ - adapted random process $w(t)$ is called a weak solution of \eqref{3.6} on $t \in [0,T]$ if $\E \sup_{t \in [0,T]} \|w(t)\|^2_{H^1} < \infty$ and for any $\f \in C^\infty(\T)$ we have
\begin{equation*}\label{1.5*}
(w(t), \f)_2 -  (w_{0}, \f)_2 =
    - \frac{1}{2}  \int_0^t ( \d_x w, \d_x \f)_2 ds
    - \int_0^t (w, \d_x \f)_2 d {\beta}(s)
\end{equation*}
$\P$ - almost surely for all $t \in [0,T]$.
\end{definition}
\begin{lemma}\label{Lem3.2}\cite{Gess1} For any $p \in [2, \infty)$ and any $w_0 \in L^p(\Omega, \F_0, \P; H^1(\T))$ there exists a unique weak solution of \eqref{3.6} on $[0,T]$ with initial value $w_0$ satisfying
\begin{equation*}
\E \left(\sup_{t \in [0,T]} \|w(t)\|_{1,2}^p + \int_0^T \|w\|_{2,2}^2 \, dt \right)< \infty,
\end{equation*}
\begin{equation}\label{3.7}
\E \left(\sup_{t \in [0,T]} \|w(t)\|_{1,2}^p \right) \leq C_1 \E  \|w_0\|_{1,2}^p,
\end{equation}
and
\begin{equation}\label{3.8}
\|\d_x w(t)\|_{2}^2 =  \|\d_x w_0\|_{2}^2,
\end{equation}
where $C_1>0$ is independent of $\w$ and $\w_0$. Furthermore,  \\
1) the mass is conserved, i.e. $\int_L w(t,x) \, dx = \int_L w \, dx$ for $t \in [0,T]$ $\P$-almost surely.\\
2) If $w_0 \geq 0$ $\P$ -almost surely, then $w(t,x,\w) \geq 0$ $\P$ -almost surely.
\end{lemma}

\begin{proof}
Apart from \eqref{3.8}, the rest of the statements follow from Proposition A.2 \cite{Gess1}.  In order to obtain \eqref{3.8}, using Ito's formula we have
\begin{multline*}
\|\d_x w(t)\|_2^2 = \int_L(\d_x w)^2 \, dx =  \int_L \d_x w(0) \, dx - \int_0^t \int_L (\d_x^2 w)^2 \, dx ds + \int_0^t \int_L (\d_x^2 w)^2 \, dx ds \\
+  2 \int_0^t \int_L \d_x w \cdot \d_x^2 w \, dx d \beta(s) =   \int_L \d_x w(0) \, dx.
\end{multline*}
\end{proof}

We now proceed with the realization of the Trotter-Kato splitting scheme. To this end, for fixed $N\geq 1$, we equi-partition the time interval $[0,T]$ into intervals of length $\delta = \frac{T}{N+1}$. Then, for $t \in [(j-1) \delta, j\delta]$ (i.e. locally) and for an arbitrary $\f \in C^{\infty} (\mathbb{\T})$ we define:\\

{\bf Deterministic Dynamics (D)} We look for the function $v_N$ which for $t \in [(j-1) \delta, j \delta]$ satisfies
\begin{equation}\label{3.9}
(v_N^\ve (t, \cdot), \f)_2 - (v_N^\ve((j-1)\delta, \cdot), \f)_2 = \int_{(j-1)\delta}^t \int_{L} f_\ve(v_N^\ve) \partial_x^3 v_N^\ve(s, \cdot) (\partial_x \f) dx ds,
\end{equation}
where $j = 1, ..., N+1$ and $\f \in C^\infty(\T)$.
\\

{\bf Stochastic Dynamics (S)} We look for the function $w_N$ satisfying
\begin{multline}\label{3.10}
(w_N(t, \cdot), \f)_2 - (w_N((j-1)\delta, \cdot), \f)_2 = - \frac{1}{2} \int_{(j-1)\delta}^t ( \partial_x w_N(s, \cdot), \partial_x \f)_2 ds \\
- \int_{(j-1)\delta}^t ( w_N(s, \cdot), \partial_x \f)_2 d \beta(s).
\end{multline}

{\bf Deterministic-Stochastic Connection (DS)} We set 
\begin{equation}\label{3.11}
v_N(0):=u_{0 \ve}, \ v_N(j \delta, \cdot) := \lim_{t \to j\delta-} w_N(t, \cdot) \text{ and }w_N((j-1) \delta, \cdot) := \lim_{t \to j\delta-} v_N(t, \cdot) \text{ a.s. }
\end{equation}

As we showed above, \eqref{3.9} and \eqref{3.10} are well posed. Furthermore, it is straightforward to verify that both of these equations conserve mass, that is 
$\int_L v_N^\ve \, dx $ and  $\int_L w_N \, dx $ are independent of time. 
In a similar way to \cite{Gess1}, we define the concatenated approximate solution $u_N^\ve:[0,T) \times \T \times \Omega \to 
[0,\infty)$ as
\begin{equation*}\label{3.12}
u_N^\ve(t,\cdot):=
\begin{cases}
v_N^\ve(2t - (\gamma-1)\delta, \cdot), \ t \in [(j-1)\delta, (j- \frac{1}{2})\delta),\\
w_N(2t - j\delta, \cdot), \ t \in [(j- \frac{1}{2})\delta, j \delta), \ \ j=1, ..., N+1.
\end{cases}
\end{equation*}
This way $u_N^\ve(t,\cdot)$ is well defined, and satisfies the mass preservation property $\int_L u_N^\ve \, dx =  \int_L u_{0,\ve} \, dx $ in $[0,T]$ $\P$- almost surely, and $u_N^\ve \geq 0 $ for every $t \in [0,T]$. Furthermore, the following proposition holds:
\begin{proposition}
For any $p \in [2, \infty)$ there exists a constant $C>0$ such that for all $N \geq 1$ and $\ve>0$ we have
\[
v_N^\ve, \ w_N^\ve \in L^p(\Omega,\mathcal{F}, \P; L^\infty([0,T); H^1(\T))
\] 
satisfying
\begin{multline}\label{3.13}
\text{ esssup }_{t \in [0,T)} \|u_N^\ve(t)\|_{1,2}^p + \text{ esssup }_{t \in [0,T)} \|v_N^\ve(t)\|_{1,2}^p + \text{ esssup }_{t \in [0,T)} \|w_N^\ve(t)\|_{1,2}^p \\
+ \int_0^t \|v_N^\ve (t)\|_{1,2}^{p-2} \int_L f_\ve(v_N^\ve) (\d_x^3 v_N^\ve)^2 \, dx dt \leq C \|u_{0 \ve}\|_{1,2}^p.
\end{multline}
\end{proposition}
\begin{proof}
By construction, $u_N^\ve$ is continuous in $t$, and
\[
\int_L u_{0, \ve} \, dx = \int_L u_N^{\ve} \, dx = \int_L v_N^{\ve} \, dx = \int_L w_N^{\ve} \, dx 
\]
for all $t \in [0, T)$, $\P$-almost surely. Next, we estimate the derivatives $\d_x v_N^\ve(j \delta)$ and $\d_x w_N^\ve(j \delta)$.  Using \eqref{3.11}  and \eqref{3.8} we have
\begin{equation*}\label{3.14}
\|\d_x v_N^\ve(\delta)\|_2^p = \lim_{t \to \delta-} \|\d_x w_N^\ve(t)\|_2^p = \|\d_x w_N^\ve(0)\|_2^p = \lim_{t \to \delta-} \|\d_x v_N^\ve(t)\|_2^p \leq \|\d_x v_{0, \ve}\|_2^p.
\end{equation*}
Using \eqref{3.5}, 
\begin{equation}\label{3.15}
\|\d_x w_N^\ve(\delta)\|_2^p = \lim_{t \to 2 \delta-} \|\d_x v_N^\ve(t)\|_2^p \leq \|\d_x v_N^\ve(\delta)\|_2^p  \leq \|\d_x v_{0, \ve}\|_2^p.
\end{equation}
Similarly,
\begin{equation*}
\|\d_x v_N^\ve(2\delta)\|_2^p = \lim_{t \to 2\delta-} \|\d_x w_N^\ve(t)\|_2^p = \|\d_x w_N^\ve(\delta)\|_2^p  \leq \|\d_x v_{0, \ve}\|_2^p
\end{equation*}
by \eqref{3.15}, and 
\begin{equation*}
\|\d_x w_N^\ve(2\delta)\|_2^p = \lim_{t \to 2 \delta-} \|\d_x v_N^\ve(t)\|_2^p \leq \|\d_x v_N^\ve(\delta)\|_2^p  \leq \|\d_x v_{0, \ve}\|_2^p.
\end{equation*}
Therefore, for any $j  = 1,..., N$ we get
\[
\| \d_x v_N^\ve(j \delta)\|_2^p \leq \| \d_x v_{0, \ve}\|_2^p \text{ and } \| \d_x w_N^\ve(j \delta)\|_2^p \leq \| \d_x v_{0, \ve}\|_2^p.
\]
We my now use \eqref{3.5} and \eqref{3.8}, combined with Poincare inequality and the conservation of mass property, we get the estimate \eqref{3.13} for $v_N^\ve$ and $w_N^\ve$, and hence for $u_N^\ve$. It remains to estimate the integral term in \eqref{3.13}. To simplify the presentation, denote
\[
 \int_0^T \|v_N^\ve (t)\|_{1,2}^{p-2} \int_L f_\ve(v_N^\ve) (\d_x^3 v_N^\ve)^2 \, dx dt := \int_0^T (\cdot) \, dt.
\]
It follows from \eqref{3.5} that
\[
2\int_0^\delta (\cdot) \, dt \leq \|\d_x  v_{0\ve}\|_2^p - 
\|\d_x v_N^\ve(\delta-0)\|_2^p,
\]
and
\begin{multline*}
2\int_\delta^{2\delta} (\cdot) \, dt \leq \|\d_x  v_{N}^{\ve}(\delta + 0)\|_2^p - 
\|\d_x v_N^\ve(2\delta-0)\|_2^p = \|\d_x  w_{N}^{\ve}(\delta - 0)\|_2^p - 
\|\d_x v_N^\ve(2\delta-0)\|_2^p \\
= \|\d_x  v_{N}^{\ve}(\delta - 0)\|_2^p - 
\|\d_x v_N^\ve(2\delta-0)\|_2^p.
\end{multline*}
Thus
\[
2\int_\delta^{2\delta} (\cdot) \, dt \leq \|\d_x  v_{0\ve}\|_2^p - 
\|\d_x v_N^\ve(2\delta-0)\|_2^p.
\]
Next, arguing similarly,
\begin{multline*}
2\int_{2\delta}^{3\delta} (\cdot) \, dt \leq  \|\d_x  v_{N}^{\ve}(2\delta + 0)\|_2^p - 
\|\d_x v_N^\ve(3\delta-0)\|_2^p = \|\d_x  w_{N}^{\ve}(2\delta - 0)\|_2^p - 
\|\d_x v_N^\ve(3\delta-0)\|_2^p \\
=\|\d_x  w_{N}^{\ve}(\delta + 0)\|_2^p - 
\|\d_x v_N^\ve(3\delta-0)\|_2^p = \|\d_x  v_{N}^{\ve}(2\delta - 0)\|_2^p - 
\|\d_x v_N^\ve(3\delta-0)\|_2^p,
\end{multline*}
hence
\[
2\int_\delta^{3\delta} (\cdot) \, dt \leq \|\d_x  v_{0\ve}\|_2^p - 
\|\d_x v_N^\ve(3\delta-0)\|_2^p.
\]
Continuing this way on the rest of the intervals, and using \eqref{3.5}, \eqref{3.7} and Poincare inequality, we get \eqref{3.13}.
\end{proof}
For any Banach space $X$ and $D \subset \R^d$, denote $B_{q}^{s,p}(D;X)$ be the Besov space, defined, e.g., in \cite{5} or \cite{56}. The space $B_{q}^{s,p}(\T;\R)$ will be denoted with $B(\T)$. 
We will now establish the analog of Proposition 4.2 \cite{Gess1}.
\begin{proposition}\label{Prop3.2}
For any $p \geq 2, \sigma > 0$, $\kappa \in (2 \sigma, \frac{2}{p}) \cap (2 \sigma, \frac{1}{2}]$ and $q \in \left(\frac{2}{\kappa - 2 \delta}, \infty\right)$, there exists $C>0$ such that for all $N \geq 1$ we have
\[
u_N^\ve \in L^p(\Omega, \F, \P; B_q^{k/2-\sigma, q} ([0,T); B_q^{1/2-2 \kappa, q}(\T)),
\]
and
\begin{equation}\label{3.16}
\E \|u_N^\ve\|^p_{B_q^{k/2-\sigma, q} ([0,T); B_q^{1/2-2 \kappa, q}(\T))} \leq C \E (\|u_{0,\ve}\|_{1,2}^p + \|u_{0,\ve}\|_{1,2}^{(\kappa+1)p}).
\end{equation}
\end{proposition}
\begin{proof}
The proposition will be established for $v_N^\ve$ and $w_N^\ve$ separately. Let us start with the estimate for $v_N^\ve$.
\begin{lemma}\label{lem3.3}
For any $p \geq 2$, $\sigma>0$ and $q \geq p$ there exists a constant $C>0$ such that for all $N \geq 1$, $1 \leq j \leq N+1$ and $\kappa \in (0, \frac{2}{p})$ we have
\[
v_N^\ve \in B_q^{k/2-\sigma, q} ([(j-1) \delta,j \delta); B_q^{1/2-2 \kappa, q}(\T)),
\]
with
\begin{equation*}\label{3.17}
\E \left(\sum_{j=1}^{N+1} \|v_N^\ve\|^p_{B_q^{k/2-\sigma, q} ([0,T); B_q^{1/2-2 \kappa, q}(\T))}\right)^{\frac{p}{q}} \leq C \E (\|u_{0,\ve}\|_{1,2}^p + \|u_{0,\ve}\|_{1,2}^{(\kappa+1)p}).
\end{equation*}
\end{lemma}
\begin{proof}
For any  $(j-1)\delta \leq t_1 \leq t_2 < j\delta$ and $\f \in C^\infty(\T)$ we have
\[
(v_N^\ve(t_2) - v_N^\ve(t_1), \f)_2  = \int_{t_1}^{t_2} \int_L f_\ve(v_N^\ve) \partial_x^3 v_N^\ve(s, \cdot) (\partial_x \f) dx ds.
\]
This $\P$ -almost surely yields
\begin{eqnarray*}
& & \|v_N^\ve(t_2) - v_N^\ve(t_1)\|^2_{H^{-1}(\T)} \leq \left( \int_{t_1}^{t_2} \left(\int_L f_\ve^2(v_N^\ve) \partial_x^3 (v_N^\ve(s, \cdot))^2 dx\right)^{1/2} ds\right)^2 \\
& &  \leq \left( \int_{t_1}^{t_2} \sup_{x \in [0,L]} |f_\ve^2(v_N^\ve)|^{1/2} \left(\int_L f_\ve(v_N^\ve) \partial_x^3 (v_N^\ve(s, \cdot))^2 dx\right)^{1/2} ds\right)^2\\
& & \leq C(t_2-t_1)  \int_{t_1}^{t_2} \|v_N^\ve\|_{1,2}  \int_L f_\ve^2(v_N^\ve) \partial_x^3 (v_N^\ve(s, \cdot))^2 dx\ ds,
\end{eqnarray*}
where we used the Sobolev embedding, Cauchy-Schwartz inequality, and the inequality $f_\ve(s) \leq s^2$. Thus, using \eqref{3.13} we have
\[
\|v_N^\ve\|_{C^{1/2}([(j-1) \delta, j \delta); H^{-1}(\T)} \leq C \|u_{0, \ve}\|_{1,2}^2
\]
$\P$ -almost surely, where $C$ is deterministic constant from Poincare inequality. The rest of the proof of Lemma \ref{lem3.3} and Proposition \ref{Prop3.2} is analogous to the proof of Proposition 4.2 \cite{Gess1}.
\end{proof}

\section{Convergence of the splitting scheme}\label{Sec5}
We start with the proposition, which is an analog of Proposition 5.2 \cite{Gess1}.
\begin{proposition}\label{Prop3.4}
Denote
    \[
\mathcal{X}_u:= BC^0([0,T] \times \T),
    \]
    \[
\mathcal{X}_J:= L^2([0,T] \times \T)\text{ (with weak topology), }
    \]
    \[
\mathcal{X}_{\beta}:= BC^0([0,T]; \R).
    \]
Then there exist random variables $\tilde{u}^{\ve}$, $\tilde{u}_N^{\ve}:[0,1] \to \mathcal{X}_u$, $\tilde{J}_N^\ve,  \tilde{J}^\ve:[0,1] \to \mathcal{X}_J$ and $\tilde{\beta}'_N, \tilde{\beta} :[0,1] \to \mathcal{X}_\beta$ with $(\tilde{u}_N^\ve, \tilde{J}_N^\ve, \tilde{\beta}'_N) \sim (u_N^\ve, J_N^\ve, \beta)$, where $J_N:= \chi_{v_N^\ve>0} f_\ve(v_N^\ve)(\d_x^3 v_N^\ve)$. Furthermore, there are subsequences (still indexed with $N$), such that $\tilde{u}_N^\ve(\omega) \to \tilde{u}^\ve(\omega)$ in  $\mathcal{X}_u$, $\tilde{J}_N^\ve(\omega) \rightharpoonup \tilde{J}^\ve(\omega)$ in $\mathcal{X}_J$ and $\tilde{\beta}'_N(\omega) \rightharpoonup \tilde{\beta}(\omega)$ in $\mathcal{X}_\beta$ for every $\omega \in [0,1]$ as $N \to \infty$. 
\end{proposition}
\begin{proof}
Apart from the convergence of $\tilde{J}_N^\ve$, the proofs of the rest of the statements are analogous to Proposition 5.2 \cite{Gess1}. In order to show the tightness of $J_N^\ve$, the Sobolev embedding and the inequality $f_\ve(s) \leq s^2$ imply
\begin{multline*}
\P\{\|J_N^\ve\|_{L^2([0,T) \times \T)} >R \} \leq \frac{1}{R} \int_0^T \E \int_L f_\ve^2 (v_N^\ve)(\d_x^3 v_N^\ve)^2 \, dx dt \\
\leq \frac{C}{R^2} \E \int_0^T \|v_N^\ve\|_{1,2}^2 \int_L f_\ve(v_N^\ve) (\d_x^3 v_N^\ve)^2 \, dx dt 
\leq \frac{C}{R^2} \E \|\d_x v_{0,\ve}\|_2^4 \to 0, \text{ as } R \to \infty,  \text{ uniformly in } N,
\end{multline*}
as well as $\{\|J_N^\ve\|_{L^2([0,T) \times \T)} \leq R \}$  is weakly compact in $L^2([0,T) \times \T)$.
\end{proof}

We will next need the analog of Proposition 5.6 from \cite{Gess1}.
\begin{proposition}\label{Prop3.5}
    Let $\tilde{u}_N^\ve, \tilde{u}^\ve, \tilde{J}_\ve^N$ and $\tilde{J}_\ve$ be as in Proposition \ref{Prop3.4}. Then the distributional derivative $\d_x^3 \tilde{u}_\ve$ satisfies $\d_x^3 \tilde{u}_\ve \in L^2(\{\tilde{u}_\ve > r\})$ for any $r>0$. Furthermore, $\tilde{J}_N^\ve = \chi_{\tilde{v}_N^\ve>0} f_\ve(\tilde{v}_N^\ve) (\d_x^3 \tilde{v}_N^\ve)$ and $\tilde{J}_\ve = \chi_{\tilde{u}_\ve>0} f_\ve(\tilde{u}_\ve) (\d_x^3 \tilde{u}_\ve)$ $\P$-almost surely. 
\end{proposition}
\begin{proof}
The proof of this proposition is similar to Proposition 5.6 \cite{Gess1} using the inequality $f_\ve(s) \leq s^2$ and monotonicity of $f_\ve(s)$ for every $\ve>0$.
\end{proof}
\begin{corollary}\label{cor_ast}
For $\tilde{u}_n^\ve$, $\tilde{v}_n^\ve$, $\tilde{w}_n^\ve$ and $\tilde{u}_\ve$ as in Proposition \ref{Prop3.4}, we have
 \begin{eqnarray}\label{ast}
 & & \|\tilde{u}_n^\ve - \tilde{u}_\ve\|_{BC^0([0,T)\times \T)} \to 0,\\
 \nonumber & & \|\tilde{v}_n^\ve - \tilde{u}_\ve\|_{L^{\infty}([0,T)\times \T)} \to 0,\\
  \nonumber & & \|\tilde{u}_n^\ve - \tilde{u}_\ve\|_{L^{\infty}([0,T)\times \T)} \to 0,
 \end{eqnarray}
 as $N \to \infty$ $\P$-almost surely. 
\end{corollary}
We may proceed with the scheme (D)-(S)-(DS) to conclude that for $t \in [0,T)$ and $\delta = \frac{T}{N+1}$ 
\begin{eqnarray*}
    & & (v_N^\ve(t), \f)_2 - (u_{0,\ve}, \f)_2 = (v_N(t),\f)_2 + \sum_{j=1}^{\left[\frac{t}{\delta}\right]} \left( -(v_N^\ve(j\delta, \cdot), \f)_2 + \lim_{s \to j\delta} (w_N^\ve(s, \cdot), \f)_2 \right) \\
    & & + \sum_{j=1}^{\left[\frac{t}{\delta}\right]}(\lim_{s \to j \delta}(v_N^\ve(s, \cdot), \f)_2 - (w_N^\ve((j-1) \delta, \cdot), \f)_2 - (v_N^\ve(0, \cdot), \f)_2 \\
    & & = (v_N^\ve(t), \f)_2  - \left(v_N^\ve \left(\left[\frac{t}{\delta}\right] \delta , \cdot\right),\f \right)_2 + \sum_{j=1}^{\left[\frac{t}{\delta}\right]}(\lim_{s \to j \delta}(v_N^\ve(s, \cdot), \f)_2 - (v_N(j-1)\delta, \cdot), \f)_2\\
    & & + \sum_{j=1}^{\left[\frac{t}{\delta}\right]}(\lim_{s \to j \delta}(w_N^\ve(s, \cdot), \f)_2 - (w_N(j-1)\delta, \cdot), \f)_2 = \int_0^t \int_{\{v_N(s,\cdot)>0\}} f_\ve(v_N^\ve)(\d_x^3 v_N^\ve)(\d_x \f) \, dx ds \\
    & & -\frac{1}{2}  \int_0^{\left[\frac{t}{\delta}\right] \delta} (\d_x  w_N^\ve(s, \cdot), \d_x \f)_2 \, ds -  \int_0^{\left[\frac{t}{\delta}\right] \delta} (w_N^\ve(s, \cdot), \d_x \f)_2 \, d\beta(s)
\end{eqnarray*}
for all $\f \in C^\infty(\T)$ $\P$ - almost surely. Changing the stochastic basis to 
\[
([0,1], \tilde{\mathcal{F}}, (\tilde{\F})_{t \in [0,T]}, \P),
\]
we obtain the existence of the in-law equivalent convergent subsequence, denoted with $\tilde{u}_N, \tilde{v}_N$ and $\tilde{w}_N$, such that
\begin{eqnarray}\label{3.18}
    & & (\tilde{v}_N(t,\cdot), \f)_2 - (u_{0,\ve}, \f)_2 =  \int_0^t \int_{\{\tilde{v}_N^\ve(s,\cdot)>0\}} f_\ve(\tilde{v}_N^\ve) (\d_x^3 \tilde{v}_N^\ve)(\d_x \f) \, dx ds \\
    \nonumber & &  -\frac{1}{2}  \int_0^{\left[\frac{t}{\delta}\right] \delta}  \d_x (\tilde{w}_N(s, \cdot)), \d_x \f)_2 \, ds  -  \int_0^{\left[\frac{t}{\delta}\right] \delta}  (\tilde{u}_N^\ve(s, \cdot)), \d_x \f)_2 \, d\tilde{\beta}(s).
\end{eqnarray}
Passing to the limit as $N \to \infty$, we get
\begin{eqnarray}\label{3.19}
    & & (\tilde{u}_\ve(t,\cdot), \f)_2 - (u_{0,\ve}, \f)_2 =  \int_0^t \int_{\{\tilde{u}_\ve(s,\cdot)>0\}} f_\ve(\tilde{u}_\ve) (\d_x^3 \tilde{v}_\ve)(\d_x \f) \, dx ds \\
    \nonumber & &  -\frac{1}{2}  \int_0^{t}  \d_x (\tilde{u}_\ve(s, \cdot)), \d_x \f)_2 \, ds  -  \int_0^{t}  (\tilde{u}_\ve(s, \cdot)), \d_x \f)_2 \, d\tilde{\beta}(s).
\end{eqnarray}
The proof of convergence in \eqref{3.18} is done analogously to Lemma 5.7 \cite{Gess1} using the estimates \eqref{3.13} and Propositions \ref{Prop3.2} through \ref{Prop3.5}.
\end{proof}
\section{The positivity of the solution.}\label{Sec6}  It follows from \eqref{3.13} that there is $C(\omega)>0$, independent on $N$, such that
\begin{equation*}\label{3.20}
\text{ esssup }_{t \in [0,T)} \|\tu_N^\ve(t)\|_{1,2}^p + \text{ esssup }_{t \in [0,T)} \|\tilde{v}_N^\ve(t)\|_{1,2}^p + \text{ esssup }_{t \in [0,T)} \|\tilde{w}_N^\ve(t)\|_{1,2}^p \leq C(\omega).
\end{equation*}
From Sobolev embedding we have
\begin{equation}\label{3.21}
0 \leq \tilde{u}_N^\ve(t,x \omega) +  \tilde{v}_N^\ve(t,x \omega) +  \tilde{w}_N^\ve(t,x \omega) \leq M_\ve(\omega)
\end{equation}
for any $(t,x) \in Q_T$ $\P$-almost surely. Define
\[
G_\ve(s):=\frac{\ve}{6s^2} - \ln(s),
\]
so that
\[
G_\ve''(s) = \frac{\ve + s^2}{s^4}.
\]
Multiplying \eqref{3.4} by $G_\ve^{'}$ and integrating by parts, we arrive at the following identity for $t \in [0,\delta)$:
\begin{equation*}\label{3.22}
\int_L G_\ve(\tilde{v}_N^\ve)\, dx + \int_0^t \int_L (\d_x^2 \tilde{v}_N^\ve)^2 \, dx \, ds = \int_L G_\ve(u_{0,\ve})\, dx.
\end{equation*}
Next, for any $\Phi \in C^2(\R)$, we apply Ito's formula to $\int_0^L \Phi(\tilde{w}_N^\ve) dx$:
\begin{eqnarray*}\label{3.23}
   \nonumber  &  & \int_L \Phi(\wN(t)) dx = \int_L \Phi(\wN(0)) dx + \frac{1}{2} \int_0^t \int_L \Phi^{'}(\wN(s))  \d_x^2(\wN(s) dx ds \\
    \nonumber  &  &  +\frac{1}{2} \int_0^t \int_L \Phi^{''}(\wN(s)) (\d_x \wN(s))^2 dx ds +  \int_0^t \int_L \Phi^{'}(\wN(s))  \d_x(\wN(s) dx d \tilde{\beta}(s) \\
    \nonumber  &  & = \int_L \Phi(\wN(0)) dx - \frac{1}{2} \int_0^t \int_L \Phi^{''}(\wN(s)) (\d_x \wN(s))^2 dx ds  \\
 \nonumber  &  & +\frac{1}{2} \int_0^t \int_L \Phi^{''}(\wN(s)) (\d_x \wN(s))^2 dx ds +  \int_0^t \int_L \frac{d}{dx}\Phi(\wN(s)) \, dx d  \beta(s) \\
 &  &  =  \int_L \Phi(\wN(0)) dx,
   \end{eqnarray*}
   due to the choice of the boundary conditions. Therefore, for $t \in [0,\delta)$, we make use of \eqref{3.11} to conclude
   \begin{equation*}\label{3.24}
   \int_L G_\ve(\wN(t))\, dx =  \int_L G_\ve(\wN(0))\, dx =  \int_L G_\ve(\vN(\delta-0))\, dx \leq  \int_L G_\ve(\tilde{u}_{0,\ve})\, dx.
   \end{equation*}
 Similarly, on $[\delta, 2 \delta)$ we obtain  
  \begin{equation*}
   \int_L G_\ve(\vN(t))\, dx =  \int_L G_\ve(\wN(\delta+0))\, dx =  \int_L G_\ve(\vN(2\delta-0))\, dx \leq  \int_L G_\ve(\tilde{u}_{0,\ve})\, dx
   \end{equation*}
   and
    \begin{equation*}
   \int_L G_\ve(\wN(t))\, dx =  \int_L G_\ve(\wN(\delta+0))\, dx =  \int_L G_\ve(\vN(2\delta-0))\, dx \leq  \int_L G_\ve(\tilde{u}_{0,\ve})\, dx.
   \end{equation*}
   Continuing this process, for all $t \in [0,T)$ and $x \in [0,L]$ we get the following inequality
     \begin{equation}\label{3.25}
   \int_L G_\ve(\vN(t))\, dx \leq  \int_L G_\ve(\tilde{u}_{0,\ve})\, dx := C(\ve, \omega)
   \end{equation}
   uniformly in $N$. Furthermore, \eqref{3.21} implies
   \[
   \ln \frac{1}{\vN} \geq \ln \frac{1}{M_\ve(\w)}.
   \]
   Since $\tilde{u}_{0,\ve} = \tilde{u}_0 + \delta(\ve)$, we have 
   \begin{equation}\label{3.26}
   \int_L \frac{dx}{(\tilde{u}_0+\delta)^2} +   \int_L \ln \frac{1}{(\tilde{u}_0+\delta)} \, dx < \infty.
   \end{equation}
   The bounds \eqref{3.25} and \eqref{3.26}yield
   \[
   \int_L \frac{\ve}{6(\vN(t,x))^2} \, dx + \int_L \ln \frac{1}{\vN(t,x)} \, dx \leq \int_L G_\ve(\tilde{u}_{0,\ve})\, dx < \infty.
   \]
    For all $\w$ and $x_0 \in [0,L]$ denote
   \[
   \delta(t,\w, \ve):= \min_{x \in L} \vN(t,x,\w) := \vN(t, x_0(t,\w),\w)
   \]
   for some $x_0 \in [0,L].$ It follows from Sobolev's embedding of $H^1(\T)$ into the Holder space $C^{1/2}(\T)$ that
   \[
   \vN(t,x,\w) \leq \vN(t,x_0,\w)+ C(\w)|x-x_0(t)|^{1/2} =  \delta(t,\w, \ve) + C(\w)|x-x_0(t)|^{1/2}.
   \]
Thus the exists a constant $K = K(\w, u_0, \ve)>0$ such that for all $N \geq 1$ we have
\begin{multline*}
K \geq \ve \int_L \frac{\ve}{(\vN(t,x))^2} \, dx \geq \ve  \int_L \frac{dx}{( \delta(t,\w, \ve) + C|x-x_0|^{1/2})^2} \\
= \ve \int_L \frac{1}{\left(1 + C\left(\frac{x-x_0}{\delta^2}\right)^{1/2}\right)^2} d\left( \frac{x-x_0}{\delta^2} \right) \geq C_1 \ve \int_0^\frac{L}{\delta^2} \frac{dy}{1+C^2y} = C_2 \ve \ln\left(1 + \frac{C^2 L}{\delta^2}\right) \geq C_3 \ve \ln \frac{1}{\delta},
\end{multline*}
which yields that $\delta(t,\w,\ve) \geq e^{-\frac{K}{C_3 \ve}}$ uniformly in $N$. Since by  \ref{ast} $\|\vN - \tu_\ve\|_{L^\infty}(Q_T) \to 0$ as $N \to \infty$, then $\tu_\ve(t,x,\w)>0$ $\tilde{P}$-almost surely. Consequently, \eqref{3.19} becomes 
\begin{multline}\label{ast}
(\tu_\ve,\f)_2 - (\tu_{0,\ve},\f)_2 = \int_0^t \int_L f_\ve (\tu_\ve) (\d_x^3 \tu_\ve) \d_x \f dx \, ds - \frac{1}{2}  \int_0^t \int_L \d_x \tu_\ve(s)  \d_x \f \, dx \, ds  \\
 -\int_0^t \int_L \tu_\ve(s) \d_x \f \, dx \, d\tilde{\beta}(s),
\end{multline}
which completes the proof of Theorem \ref{Thm3.1}. 
\end{proof}

\section{Global solutions of the regularized problem and their asymptotic behavior.}\label{Sec7}
In this section we show that the local solution, obtained in Theorem \ref{Thm3.1}, can be extended for all $t \geq 0$. Throughout this section, in order to simplify the notation, we will omit the tildes and denote $\Omega$ to be $[0,1]$.
\begin{theorem}\label{Thm4.1}
Under the conditions of Theorem \ref{Thm3.1}, the solution of \eqref{3.1} in the sense of Definition \ref{Def1.1} with the initial condition \eqref{3.2} exists for $t \in [0, \infty)$. 
\end{theorem} 
\begin{proof}
In view of Theorem \ref{Thm3.1}, \eqref{3.1}  we have a filtered probability space $(\Omega^1 = [0,1], \F^1, \F_t^1, \P^1)$, on which we have a martingale solution $u_\ve^1(t)$ on $[0,T)$, with the Wiener process $\beta_1(t)$ and filtration $\F_t^1 = \sigma\{(u_{0,\ve}, \beta_1(s)), s \leq t, t \in [0,T) \}$. This solution is strictly positive, and \eqref{3.25}, the convergence $\|v_N^\ve - u_\ve^1\|_{L^\infty(Q_T)} \to 0$, $N \to \infty$, and Fatou Lemma yield
\begin{equation*}\label{3.27}
\int_L G_\ve(u_\ve^1(T))\, dx  \leq C(\ve, \w).
\end{equation*}
Furthermore, it follows from \eqref{3.13} that $\|u_\ve^1(T)\|_{1,2} \leq \|u_{\ve,0}\|_{1,2}$ for all $p \geq 2$. We next define a new probability space
$(\Omega', \F', \P')$ and introduce a new Wiener process $\beta'(t,\w')$ on it.  Finally, introduce the product probability space  $$(\Omega^2, \F^2, \P^2):=(\Omega^1 \times \Omega', \F^1 \times \F', \P^1 \times \P'),$$
and introduce the Wiener process 
\[
\beta_2(t, \w_1, \w'): = \begin{cases}
\beta_1(t, \w_1), t \in [0,T);\\
\beta'(t-T, \w') + \beta_1(T,\w_1), t \in [T, 2T].
\end{cases}
\]
For every fixed $\w_1$ the equation \eqref{3.1}  has a martingale solution $\eta_\ve(t,\w')$ on $[T, 2T)$ with the initial condition $\eta_\ve(T,\w') = u_\ve^1(T,\w_1)$. The same way as in \eqref{3.25} we conclude 
  \begin{equation*}
   \int_L G_\ve(\eta_\ve(t,\w'))\, dx \leq  \int_L G_\ve(\eta_\ve(T,\w'))\, dx =  \int_L G_\ve(u_\ve^1(T,\w_1))\, dx.
   \end{equation*}
   Hence
   \begin{equation*}\label{3.28}
   \int_L G_\ve(\eta_\ve(2T,\w'))\, dx \leq  \int_L G_\ve(u_\ve^1(T,\w_1))\, dx.
   \end{equation*}
   We may now define the solution
   \[
   u_\ve^2(t,\w_1, \w') :=  \begin{cases}
u^1_\ve(t, \w_1), t \in [0,T);\\
\eta_\ve(t, \w'), t \in [T, 2T]
\end{cases}
   \]
   on the space $(\Omega^2, \F^2, \F_t^2, \P^2)$, where the filtration $ \F_t^2$ is defined as
   \[
    \F_t^2 = \sigma\{u_{0,\ve}, \beta_2(s), u_\ve^2(s), s \leq t\}, t \in [0,2T).
   \]
   Repeating this process $n$ times, we obtain a continuous in $t$ solution on $[0, nT]$ for any $n \geq 1$. Their finite dimensional distributions would match by construction. It remains to use Kolmogorov's theorem, which guarantees the existence of a probability space and  a weak solution $u_\ve(t), t \geq 0$, defined on it. This completes the proof of Theorem \ref{Thm4.1}.
\end{proof}

The solution, obtained in Theorem \ref{Thm4.1}, is strictly positive. Moreover, since $G_\ve''(s) = \frac{1}{f_\ve(s)}$, by Ito's formula we have
\begin{multline*}
\int_L G_\ve(u_\ve(t))\, dx = \int_L G_\ve(u_{0,\ve})\, dx + \int_0^t \int_L G_\ve''(u_\ve) f_\ve(u_\ve) (\d^3_x u_\ve) \d_x u_\ve\, dx \, ds \\
-  \frac{1}{2} \int_0^t \int_L G_\ve''(u_\ve) (\d_x u_\ve)^2 \, dx \, ds
+  \frac{1}{2} \int_0^t \int_L G_\ve''(u_\ve) (\d_x u_\ve)^2 \, dx \, ds + \ \int_0^t \int_L G_\ve'(u_\ve) \d_x u_\ve \, dx \, d \beta(s).
\end{multline*}
The latter identity yields
\begin{equation}\label{3.29}
\int_L G_\ve(u_\ve(t))\, dx  + \int_0^t \int_L  (\d^2_x u_\ve)^2 \, dx \, ds = \int_L G_\ve(u_{0,\ve})\, dx,
\end{equation}
thus $u_\ve \in H^2(\T)$ for all $t \geq 0$ a.s. 

\begin{lemma}\label{lem4.2}
The function $\|\d_x u_\ve\|_2^2$ is non-decreasing for $t \geq 0$ almost surely.
\end{lemma}
\begin{proof}
Using Ito's formula for $\|\d_x u_\ve\|_2^2$ we have 
\begin{multline*}
\| \d_x u_\ve(t)\|_2^2 =\|\d_x u_{0,\ve} \|_2^2 + 2 \int_0^t \int_L \d_x(f_\ve(u_\ve) (\d^3_x u_\ve)) \d_x^2 u_\ve\, dx \, ds \\
-  \int_0^t \int_L  (\d_x^2 u_\ve)^2 \, dx \, ds
+   \int_0^t \int_L (\d_x^2 u_\ve)^2 \, dx \, ds.
\end{multline*}
This way
\begin{equation}\label{3.30}
\| \d_x u_\ve(t)\|_2^2 +2 \int_0^t \int_L f_\ve(u_\ve) (\d^3_x u_\ve))^2 \, dx \, ds =  \|\d_x u_{0,\ve} \|_2^2  =  \|\d_x u_{0} \|_2^2
\end{equation}
which completes the proof.
\end{proof}
Furthermore, using the Sobolev embedding, Poincare inequality and the conservation of mass property we have 
\[
u_\ve(t,x) - \overline{u_{0,\ve}} \leq C \int_L (\d_x u_\ve)^2 \, dx \leq C \|\d_x u_0\|^2,
\]
which, in turn, yields the following corollaries.
\begin{corollary}\label{cor4.3}
There is a nonrandom constant $K>0$, independent of $\ve$, such that for all $x \in [0,L]$ and $t \geq 0$ we have
\begin{equation}\label{3.31}
0<u_\ve(t,x) \leq K \text{ almost surely.}
\end{equation}

\end{corollary}

\begin{corollary}\label{cor4*}
For any $p \geq 2$, by the same reasoning as in Lemma \ref{lem4.2},  there is $C>0$ such that for all $t \geq 0$ we have
\begin{equation*}\label{3.3*}
\sup_{s \in [0,t]}\| u_\ve(s)\|_{1,2}^p +2 \int_0^t\| u_\ve(s)\|_{1,2}^{p-2} \int_L f_\ve(u_\ve) (\d^3_x u_\ve))^2 \, dx \, ds \leq C \|\d_x u_{0,\ve} \|_{1,2}^p.
\end{equation*}
\end{corollary}
The bound \eqref{3.31} implies that the $\ln \frac{1}{u_\ve(t)}$ is bounded from below for all $t \geq 0$ and $x \in [0,L]$ $\P$ - almost surely. Hence $\int_L G_\ve(u_\ve(t))\, dx$ is bounded from above uniformly in $\ve \in (0,1)$. As for the upper bound for this quantity, using \eqref{3.29} as well as \eqref{2.1}, we have
\begin{multline*}
\int_L G_\ve(u_\ve(t))\, dx \leq \int_L G_\ve(u_{0,\ve})\, dx = \int_L \frac{\ve \, dx}{6(u_0+\delta(\ve))^2}  + \int_L \ln \left(\frac{1}{u_0+\delta(\ve)}\right) \, dx \\
\leq L \ve^{1-2\theta} +  \int_L \ln \left(\frac{1}{u_0}\right) \, dx \leq L +  \int_L \ln \left(\frac{1}{u_0}\right) \, dx < \infty.
\end{multline*}
This leads us to the following corollary:
\begin{corollary}\label{cor4.4}
There is a nonrandom constant $C>0$, independent of $\ve$, such that 
\begin{equation}\label{3.32}
\int_0^\infty \int_L  (\d_x^2 u_\ve)^2 \, dx \, ds < C.
\end{equation}

\end{corollary}

We now proceed with establishing the estimate \eqref{1.6}. To this end, we will modify the approach from \cite{Tud} to stochastic case. Firstly, note that due to the choice of periodic boundary conditions, we have the analogs of Lemmas 3 and 4 from \cite{Tud}:
\begin{lemma}\label{lem4.5}
There is $C>0$ such that for any $\beta \in \R$ and any strictly positive $v \in H^2(\T)$ we have
\begin{equation*}\label{3.33}
\int_L v^\beta (\d_x^2 v)^2 \, dx \geq C (1-\beta)^2 \int_L v^{\beta-2} (\d_x v)^4 \, dx. 
\end{equation*}
\end{lemma}

\begin{lemma}\label{lem4.6}
There is $C>0$ such that for any $v \in H^2(\T)$ with positive mass we have
\begin{equation*}\label{3.34}
\int_L v^2 (\d_x^2 v)^2 \, dx \geq C  \left(\int_L v(x)\, dx \right)^2 \int_L(\d_x v)^2 \, dx. 
\end{equation*}
\end{lemma}
Note also that in view of \eqref{3.31} and Corollary \ref{cor4.4}, the solution $u_\ve(t,x,\w)$ satisfies the conditions of Lemma \ref{lem4.5} and Lemma \ref{lem4.6} $\P$-almost surely. The next proposition pretty much follows the lines of Theorem 3 \cite{Tud}. Nevertherless, for the convenience of the reader, we shall highlight the key ingredients of the proof below.

\begin{proposition}
Assume $u_0 \in H^1(\T)$ and the condition \eqref{2.1} holds. Then
\[
J_\ve[u_\ve(t)]:= \frac{1}{2} \int_L (\d_x u_\ve(t))^2\, dx
\]
satisfies
\begin{equation}\label{3.35}
J_\ve[u_\ve(t)] \leq e^{-K_\ve^2 t} \int_L (\d_x u_0(t))^2\, dx + C \ve^{1/2} K_\ve \int_0^t   \int_L (\d_x u_\ve(s))^2\, dx ds \text{ a.s., }
\end{equation}
where $C>0$ is a non-random constant, independent of $\ve$, and $K_\ve>0$ is a deterministic constant satisfying
\begin{equation}\label{astast}
\lim_{\ve \to 0} K_\ve = K < \infty.
\end{equation}
\end{proposition}
\begin{proof}
For $t \geq 0$ denote 
\[
I_\ve(t):=  \int_L f_\ve(u_\ve) (\d^3_x u_\ve))^2 \, dx.
\]
By Schwartz inequality,
\[
(2 I_\ve J_\ve)^{1/2} \geq - \int_L f_\ve^{1/2}(u_\ve) (\d_x u_\ve)(\d_x^3 u_\ve) \, dx.
\]
The same way as in \cite{Tud}, Theorem 3, we get
\[
(2 I_\ve J_\ve)^{1/2} \geq - \int_L f_\ve^{1/2}(u_\ve) (\d_x^2 u_\ve)^2 \, dx -\frac{2}{3} \ve^2 \int_L (\ve+ u_\ve^2)^{-\frac{5}{2}} (\d_x u_\ve)^4 \, dx.
\]
Thus, from \eqref{3.31} we have
\[
f_\ve^{1/2}(u_\ve) \geq \frac{u_\ve^2}{(\ve + K^2)^{\frac{1}{2}}}
\]
for almost all $t \geq 0$ and $x \in [0,L]$. Then from Lemma \ref{lem4.5} and Lemma \ref{lem4.6} there exist $C>0$ and $C_1>0$ such that
\begin{equation}\label{3.36}
C_1 K_\ve^2 J_\ve \leq I_\ve + C \ve^{1/2} K_\ve \int_L(\d_x^2 u_\ve)^2\, dx,
\end{equation}
where 
\[
K_\ve = \frac{\pi^2 (1+2 \ve^\theta)^2}{16(\pi+1)^2 (\ve+K^2)^{1/2}}.
\]
From  \eqref{3.30} we thus get
\[
\frac{d}{dt} J_\ve[u_\ve(t)] = - \int_L f_\ve(u_\ve) (\d_x^3 u_\ve)^2 \, dx.
\]
The latter equality, combined with \eqref{3.36}, implies 
\[
\frac{d}{dt} J_\ve[u_\ve(t)] \leq -C_1 K_\ve^2 J_\ve[u_\ve] + C \ve^{1/2} K_\ve \int_L (\d_x^2 u_\ve)^2 \, dx,
\]
and thus \eqref{3.35} follows from the Gronwall inequality. 
\end{proof}

\section{Proof of the main result.}\label{Sec8}
\subsection{Compactness and convergence of a subsequence.}
In Theorem \ref{Thm4.1} we constructed a weak solution $u_\ve$ of the equation \eqref{3.1} on $[0, \infty)$ for any $\ve>0$. Each of these solutions exists on its own probability space $(\Omega_\ve, \F_\ve,{\mathcal{F}}_{t,\ve},  \P_\ve).$ In order to pass to the limit as $\ve \to 0$, we will take a sequence $\ve = \frac{1}{n}$ and consider $(\beta_\ve,u_\ve)$ on the common probability space 
\[
(\Omega, \F, \F_t, \P)  := \Pi_{\ve} (\Omega_\ve, \F_\ve,\F_{t,\ve}, \P_\ve).
\]
Fix $T>0$. We shall pass to the limit as $\ve \to 0$ in \eqref{3.1} on $[0,T]$ using the techniques from \cite{Gess2}. Note, however, that in \cite{Gess2} the authors used the regularization $f_\ve(s) = (s^2 +\ve^2)^{n/4}, n\in [8/3,4)$ while in this work we use $f_\ve(s) = \frac{s^4}{s^2 +\ve}$.

\begin{lemma}\label{lem5.1}
Suppose $\ve \in (0,1]$,  $p>1$ and $u_0 \in L^{4p}(\Omega, \mathbb{F}_0, \P; H^1(\T))$ satisfies $u_0 \geq 0$ $\P$-almost surely, and
\[
\E \left(\int_L u_0(x) \, dx\right)^{4p} < \infty.
\]
Then, for any $p' \in [1, 2p)$, the weak solution $u_\ve$, whose existence was established in Theorem \ref{Thm4.1}, satisfies 
\begin{equation*}\label{5.1}
u_\ve \in L^{p'}(\Omega, \F, \P; C^{1/4} ([0,T]; L^2(\T)),
\end{equation*}
with
\begin{equation}\label{5.2}
\|u_\ve\|_{ L^{p'}(\Omega, \F, \P; C^{1/4} ([0,T]; L^2(\T))} \leq C\left[ \left( \E \int_L u_0(x,\w) \, dx)^{2p'}\right) + \left(\E \|\d_x u_0\|_2^{2p}\right)^{p'/2p} + \E \|\d_x u_0\|_2^{2p'}\right],
\end{equation}
where the constant $C$ depends only on $p, p', L$ and $T$.
\end{lemma} 
\begin{proof}
It follows from \eqref{ast} that for any $\f \in H^1(\T)$, and any $t_1 \leq t_2 \in [0,T]$, the identity
\begin{multline*}
(u_\ve(t_2)- u_\ve(t_1),\f)_2 +  \int_{t_1}^{t_2} \int_L f_\ve (u_\ve) (\d_x^3 u_\ve) \d_x \f dx \, ds + \frac{1}{2}  \int_{t_1}^{t_2} \int_L \d_x u_\ve(s)  \d_x \f \, dx \, ds  \\
 = \int_{t_1}^{t_2} \int_L \d_x u_\ve(s)  \f \, dx \, d \beta(s)
\end{multline*}
holds $\P$-almost surely. Since $H^1(\T)$ is separable, using the approximation argument we may conclude that the $\P$-zero set can be chosen independently of $\f$, in other words, $\P$-almost surely we have the following bound
\begin{multline}\label{5.4}
(u_\ve(t_2)- u_\ve(t_1),\f)_2 +  \int_{t_1}^{t_2} \int_L f_\ve (u_\ve) (\d_x^3 u_\ve) \d_x \f dx \, ds + \frac{1}{2}  \int_{t_1}^{t_2} \int_L \d_x u_\ve(s)  \d_x \f \, dx \, ds  \\
 \leq \sup_{\|\psi\|_2 \leq 1} |(I_\ve(t_2)-I_\ve(t_1),\psi)_2|
\end{multline}
for all $\f \in H^1(\T)$ with $\|\f\|_2 =1$, where 
\[
I_\ve(t) =  \int_{0}^{t} \d_x u_\ve(s)  \, d \beta(s).
\]
Setting 
\[
\f:= \frac{u_\ve(t_2)- u_\ve(t_1)}{\|u_\ve(t_2)- u_\ve(t_1)\|_2}
\]
in \eqref{5.4}, we have
\begin{eqnarray*}
& &\frac{\|u_\ve(t_2)- u_\ve(t_1)\|_2^2}{\|u_\ve(t_2)- u_\ve(t_1)\|_2} \leq \frac{1}{\|u_\ve(t_2)- u_\ve(t_1)\|_2} \left| \int_{t_1}^{t_2} \int_L f_\ve (u_\ve) (\d_x^3 u_\ve) \d_x (u_\ve(t_2)- u_\ve(t_1)) dx \, ds \right| \\
& & + \frac{1}{2 \|u_\ve(t_2)- u_\ve(t_1)\|_2} \left| \int_{t_1}^{t_2} \int_L \d_x u_\ve(s)  \d_x(u_\ve(t_2)- u_\ve(t_1)) \, dx \, ds \right| +\| I_\ve(t_2) - I_\ve(t_1)\|_2.
\end{eqnarray*}
Multiplying the latter inequality with $\|u_\ve(t_2)- u_\ve(t_1)\|_2$ and applying Young's inequality, we get that there is a $C>0$, independent of $\ve$, such that 

\begin{eqnarray}\label{5.5}
& &\|u_\ve(t_2)- u_\ve(t_1)\|_2^2 \leq C \left| \int_{t_1}^{t_2} \int_L f_\ve (u_\ve) (\d_x^3 u_\ve) \d_x (u_\ve(t_2)- u_\ve(t_1)) dx \, ds \right| \\
\nonumber & & + \left| \int_{t_1}^{t_2} \int_L \d_x u_\ve(s)  \d_x(u_\ve(t_2)- u_\ve(t_1)) \, dx \, ds \right| +\| I_\ve(t_2) - I_\ve(t_1)\|_2^2 \\
\nonumber & & := C (J_1^2(t_1,t_2)+J_2^2(t_1,t_2)) + \| I_\ve(t_2) - I_\ve(t_1)\|_2^2.
\end{eqnarray}
Let us now estimate each term separately. Using Lemma 2.1 \cite{Flan}, for $p \geq 2$ and $\alpha<\frac{1}{2}$ we have for any $T>0$
\begin{equation}\label{5.6}
\E \left\|\int_0^T   \d_x (u_\ve(s))  d \beta(s) \right\|^p_{W^{\alpha,p}([0,T], L^2(\T))} \leq C(p,\alpha) \E  \left(\int_0^T \|\d_x u_\ve\|_2^p \, ds  \right) \leq C(p,\alpha) T \E  \|\d_x u_0\|_2^p,
\end{equation}
where we used \eqref{3.30}. From \eqref{5.6} and the continuous embedding
\[
W^{\alpha, p}([0,T], L^2(\T)) \subset C^{\alpha - \frac{1}{p}}([0,T];L^2(\T))
\]
for all $\alpha \in (0,\frac{1}{2})$, we get 
\begin{equation*}\label{5.7}
\E \|I_\ve\|_{L^{p'}(\Omega, C^{\frac{1}{2} - \sigma}([0,T]; L^2(\T))} \leq C \left(\E \|\d_x u_0 \|_2^{2p}\right)^\frac{1}{2p}
\end{equation*}
for all $0<\sigma <\frac{1}{2}$, $1 \leq p' < 2p$ and $p \geq 2$ satisfying the assumptions of Lemma 2.1  \cite{Flan}. Therefore, for such $\sigma$ and $p'$, 
\begin{equation}\label{5.8}
\E \sup_{t_1,t_2 \in [0,T]} \left(\frac{|I_\ve(t_2) - I_\ve(t_1)|}{|t_2-t_1|^{1/2-\sigma}}\right)^{p'} \leq C  \left(\E \|\d_x u_0\|_2^{2p} \right)^\frac{p'}{2p}.
\end{equation}
Let us estimate $J_1$ in \eqref{5.5}. We have
\begin{multline*}
J_1(t_1,t_2)  =  \left( \int_{t_1}^{t_2} \int_L f_\ve (u_\ve) (\d_x^3 u_\ve) \d_x (u_\ve(t_2)- u_\ve(t_1)) dx \, ds \right)^{\frac{1}{2}} \\
\leq  \left( \int_{t_1}^{t_2} \int_L f_\ve (u_\ve) (\d_x (u_\ve(t_2)- u_\ve(t_1)))^2 \, dx \, ds \right)^{\frac{1}{4}}    \left( \int_{t_1}^{t_2}  \int_L  f_\ve (u_\ve) (\d_x^3 u_\ve)^2 \, dx \, ds \right)^{\frac{1}{4}}.
\end{multline*}
Using the inequality $f_\ve(s) \leq s^2$, \eqref{3.30}, \eqref{3.31} and using the conservation of mass property, we have
\begin{eqnarray*}
& & J_1(t_1,t_2) \leq C |t_2-t_1|^{\frac{1}{4}} \left(1+ \|\d_x u_0\|_2 + \frac{1}{L} \int_{L} u_{0,\ve}\, dx\right) \left( \int_{0}^{T} \int_L f_\ve (u_\ve) (\d_x^3 u_\ve)^2 \, dx \, ds \right)^{\frac{1}{4}}  \\
\nonumber & & \leq  C_1 |t_2-t_1|^{\frac{1}{4}} \left(1+ \|\d_x u_0\|_2 + \frac{1}{L} \int_{L} u_{0,\ve}\, dx\right) \|\d_x u_0\|_{2}^{\frac{1}{2}} \\
 \nonumber & & \leq C_2 |t_2-t_1|^{\frac{1}{4}} \left(1+ \|\d_x u_0\|_2^2 + (\overline{u_{0,\ve}})^2\right).
\end{eqnarray*}
Hence,
\begin{equation}\label{5.9}
\E \sup_{t_1, t_2 \in [0,T]} \left(\frac{J_1(t_1,t_2)}{|t_2-t_1|^{\frac{1}{4}}}\right)^{p'} \leq C_3  \E \left(1+ \|\d_x u_0\|_2^2 + (\overline{u_{0,\ve}})^2\right).
\end{equation}
In order to estimate $J_2$, using \eqref{3.30}, we have
\begin{multline*}
2 J_2(t_1, t_2) = \left| \int_{t_1}^{t_2} \int_L \d_x u_\ve(s)  \d_x(u_\ve(t_2)- u_\ve(t_1)) \, dx \, ds \right|^{\frac{1}{2}}\\
 \leq C T^{\frac{1}{4}} |t_2-t_1|^{\frac{1}{4}} \left(\sup_{t\in [0,T]} \|\d_x u_
\ve(t)\|_2^2\right)^\frac{1}{2} \leq C_1  |t_2-t_1|^{\frac{1}{4}} \|\d_x u_0\|_2.
\end{multline*}
Therefore, 
\begin{equation}\label{5.10}
\E \sup_{t_1, t_2 \in [0,T]} \left(\frac{J_1(t_1,t_2)}{|t_2-t_1|^{\frac{1}{4}}}\right)^{p'} \leq C  \E  \|\d_x u_0\|_2^{p'}.
\end{equation}
Combining \eqref{5.8}, \eqref{5.9} and \eqref{5.10}, and choosing $\delta = \frac{1}{4}$, we get
\begin{multline*}
\|u_\ve\|^{p'}_{L^{p'}(\Omega, C^{\frac{1}{4}}([0,T]; L^2(\T)))} \leq \tilde{C} \left[ \left(\E \|\d_x u_0\|_2^{2p}\right)^{\frac{p'}{2p}} + 1 + \E \|\d_x u_0\|_2^{2p'} + \E (\overline{u_{0,\ve}})^{2p'}+ \E  \|\d_x u_0\|_2^{p'}\right]\\
\leq \tilde{C}_1 \left[1 +  \left(\E \|\d_x u_0\|_2^{2p}\right)^{\frac{p'}{2p}} +  \E \|\d_x u_0\|_2^{2p'} + \E (\overline{u_{0}})^{2p'} \right],
\end{multline*}
where the constant $\tilde{C}_1 > 0$ is independent of $\ve$. Similarly to \cite{Gess2}, we apply the interpolation argument, \eqref{5.2} and the continuous embedding $H^1(\T)$ into $C^{1/2}(\T)$ we have 
\begin{corollary}\label{cor5.2}
Under the conditions of Lemma \ref{lem5.1}, the solution $u_\ve$ in Holder continuous $\P$-almost surely. In particular, there is a constant $C>0$, independent of $\ve>0$, such that
\begin{equation}\label{5.11}
\E \left(\|u_\ve\|^p_{C^{\frac{1}{8}, \frac{1}{2}}(Q_T)}\right) \leq C
\end{equation}
for any $p' \in [1,2p).$ 
\end{corollary}

Denote
\[
C^{\frac{1}{8}-, \frac{1}{2}-}(Q_T):= \cap_{\gamma \in \left(0, \frac{1}{2}\right)} C^{\frac{\gamma}{4}, \gamma}(Q_T).
\]
We endow $C^{\frac{1}{8}-, \frac{1}{2}-}(Q_T)$ with the topology generated by the metric 
\[
\rho(u,v):= \sum_{n=1}^{\infty} 2^{-n}(\|u-v\|_{C^{\frac{\gamma_n}{4}, \gamma_n}(Q_T)} \wedge 1), 
\]
where $\gamma_n = \frac{1}{2} - \frac{1}{2^{n+1}}$.  The space $C^{\frac{1}{8}-, \frac{1}{2}-}(Q_T)$ is a Polish space, moreover, the embedding $C^{\frac{1}{8}, \frac{1}{2}}(Q_T) \subset C^{\frac{1}{8}-, \frac{1}{2}-}(Q_T)$ is compact, as shown in \cite{Gess2}. The following proposition is the analog of Proposition 5.4 \cite{Gess2}. 

\begin{proposition}\label{prop5.3}
Let $T \in [0, \infty), \ \ve \in (0,1], \ p>1$ and $u_0 \in L^{4p}(\Omega, \F_0, \P; H^1(\T))$ be such that $u_0 \geq 0$ $\P$-almost surely. Theorem \ref{Thm4.1} guarantees the existence of a weak solution  $u_\ve$ with such initial condition. Define $J_\ve:= f_\ve(u_\ve) \d_x^3 u_\ve$. Then there is a subsequence in $\ve$, and the random variables 
\[
\tu, \tu_\ve: [0,1] \to C^{\frac{1}{8}-, \frac{1}{2}-}(Q_T),
\]
\[
\tilde{J}, \tilde{J}_\ve:[0,1] \to L^2(Q_T),
\]
\[
\tilde{\beta}, \tilde{\beta}_\ve:[0,1] \to C([0,T], \R),
\]
with
\begin{equation*}\label{5.12}
(\tilde{u}_\ve, \tilde{J}_\ve, \tilde{\beta}_{\ve}) \sim (u_\ve, J_\ve, \beta_{\ve}),
\end{equation*}
such that
\begin{equation}\label{5.13a}
\tu_\ve \to \tu \text{ as } \ve \to 0 \text{ in } C^{\frac{1}{8}-, \frac{1}{2}-}(Q_T),
\end{equation}
\begin{equation}\label{5.13b}
\tilde{J}_\ve \rightharpoonup \tilde{J} \text{ as } \ve \to 0 \text{ in } L^2(Q_T)
\end{equation}
and
\begin{equation*}\label{5.13c}
\tilde{\beta}_\ve \to \tilde{\beta} \text{ as } \ve \to 0 \text{ in }  C([0,T], \R)
\end{equation*}
for every $\w \in [0,1]$. Furthermore, 
\begin{equation}\label{5.14}
\tu \in L^p(\tilde{\Omega}, \tilde{\F}, \tilde{\P}; C^{\frac{\gamma}{4}, \gamma}(Q_T)) \text{ for all } \gamma \in \left(0,\frac{1}{2}\right) \text{ and }  p' \in [1, 2p).
\end{equation}
\end{proposition}
\begin{proof}
By Skorokhod's Theorem \cite{Skor} and its generalization \cite{Ykub} it suffices to show the tightness of the measures corresponding to the process $u_\ve, J_\ve$ and $\beta_\ve$. The tightness of the law $\mu_{\beta_\ve}$ follows from the fact that it is a Radon  measure in the Polish space $C([0,T]; \R)$. The tightness of the law for $\mu_{u_\ve}$ on $C^{\frac{1}{8}-, \frac{1}{2}-}(Q_T)$ follows from \eqref{5.11} and the compactness of the embedding $C^{\frac{1}{8}, \frac{1}{2}}(Q_T) \subset C^{\frac{1}{8}-, \frac{1}{2}-}(Q_T).$ The tightness of the law for $J_\ve$ follows from \eqref{3.30}. Finally, \eqref{5.14} follows from \eqref{5.13a}, \eqref{5.11} and Fatou's Lemma, which completes the proof of Proposition \ref{prop5.3}. Furthermore, from \eqref{5.13a} and the positivity of $\tu_\ve(t,x,\w)$ we have $\tu(t,x,\w) \geq 0$ $\P$-almost surely. 
\end{proof}
\begin{proposition}\label{prop5.4}
Let $T \in (0,\infty)$, $\ve \in (0,1]$ and $u_0 \in L^2(\Omega, \tilde{\F}_0, \P, H^1(\T))$ is such that $u_0\geq 0$ $\P$-almost surely. Then there is a subsequence, still denoted $\tilde{u}_\ve$, such that
\[
\d_x \tu_\ve \rightharpoonup \d_x \tu \text{ in } L^2(\tilde{\Omega}, \tilde{\F}, \tilde{\P}, L^2([0,T) \times \T)) \text{ as } \ve \to 0. 
\]
Furthermore,  the distributional derivative $\d_x^3 \tu$ is in $L^2(\{\tu > r\})$, and we can identify $\tilde{J}_\ve = f_\ve(\tu_\ve) \d_x^3 \tu_\ve $ and $\tilde{J} = \chi_{\{\tu>0\}} \tu^2 \d_x^3 \tu$ \ $\P$-almost surely.
\end{proposition}
\begin{proof}
It follows from \eqref{3.30} we have  $\d_x \tu_\ve$ is weakly compact in $L^2(\tilde{\Omega}, \tilde{\F}_0, \P, L^2([0,T) \times \T))$. Thus, for any $\f \in L^{\infty}(\tilde{\Omega}, \tilde{\F}_0, \P; C^1(Q_T))$  we have 
\begin{equation}\label{5.15}
\tilde{E} \int_{Q_T} \d_x \tu_\ve \f \, dx dt \to \tilde{E} \int_{Q_T} \eta \cdot  \f \, dx dt \text{ as } \ve \to 0.  
\end{equation}
On the other hand, it follows from \eqref{3.31}, \eqref{5.13a}, dominated convergence theorem, as well as from the fact that $u_\ve$ and $\tu_\ve$ have the same distribution, that
\[
\tilde{E} \int_{Q_T} \tu_\ve  \d_x \f \, dx dt \to \tilde{E} \int_{Q_T} \tu  \d_x \f \, dx dt \text{ as } \ve \to 0.  
\]
Combining the latter convergence with \eqref{5.15}, we get $\eta = \d_x \tu$ $\P$-almost surely. Similarly, since the joint distribution laws of $(u_\ve, J_\ve)$ and $(\tu_\ve, \tilde{J}_\ve)$ are the same, for any $\f \in C^{\infty}(Q_T)$,
\[
0 = \E \left|\int_{Q_T} J_\ve \, \f dx  dt - \int_{Q_T} f_\ve(u_\ve) (\d_x^3 u_\ve) \, dx dt\right| = \tilde{\E} \left|\int_{Q_T} \tilde{J}_\ve \, \f dx  dt - \int_{Q_T} f_\ve(\tu_\ve) (\d_x^3 \tu_\ve) \, dx dt\right|,
\]
which implies that $\tilde{J}_\ve = f_\ve(\tu_\ve) (\d_x^3 \tu_\ve)$. Furthermore, since $f_\ve(s)$ is non-decreasing in $s$, for any $r>0$, using \eqref{3.30} we have
\begin{equation*}\label{5.16}
\tilde{\E} \int_{Q_T} (\d_x^3 u_\ve)^2 \chi_{\{\|\tu_\ve - \tu\|_{L^{\infty}(Q_T)} < \frac{r}{2}\} \cap \{\tu > r\}} \, dx  dt \leq \frac{1}{f_\ve(\frac{r}{2})} \tilde{\E} \int_0^T \int_{\tu_\ve(t) > \frac{r}{2}} f_\ve(\tu_\ve) (\d_x^3 \tu_\ve)^2 \, dx \, dt \leq C(r,u_0),
\end{equation*}
where $C(r,u_0)< \infty$ is independent of $\ve$. Thus, up to a subsequence, still denoted with $\tu_\ve$, we have
\begin{equation}\label{5.17}
\d_x^3 u_\ve \chi_{\{\|\tu_\ve - \tu\|_{L^{\infty}(Q_T)} < \frac{r}{2}\} \cap \{\tu > r\}} \, dx  dt \rightharpoonup \eta(r) \chi_{\{\tu > r\}}
\end{equation}
as $\ve \to 0$ in $L^2(\tilde{\Omega}, \tilde{\F}, \tilde{\P}; L^2(Q_T))$. In order to identify
$\eta(r)$, for any $\f \in L^2(\tilde{\Omega}, C^\infty(Q_T))$ with $\text{supp}_{(t,x) \in Q_T} \f \subset \{\tu>r\}$ for a.e. $\tilde{\w}$, 
\begin{multline}\label{5.17_1}
\tilde{\E} \int_{Q_T} (\d_x^3 u_\ve) \chi_{\{\|\tu_\ve - \tu\|_{L^{\infty}(Q_T)} < \frac{r}{2}\} \cap \{\tu > r\}} \f \, dx  dt = -\tilde{E} \int_{Q_T} u_\ve \chi_{\{\|\tu_\ve - \tu\|_{L^{\infty}(Q_T)} < \frac{r}{2}\} \cap \{\tu > r\}} \d_x^3 \f \, dx  dt \\
\to -\tilde{\E} \int_{Q_T} \tu \chi_{\{\tu > r\}} \d_x^3 \f \, dx  dt 
\end{multline}
as $\ve \to 0$, for any $r>0$. The limit in \eqref{5.17_1} follows from the Dominated Convergence theorem, \eqref{3.31}, the fact that $u_\ve$ and $\tu_\ve$ have the same distributions, as well as the fact that the constant $K$ in \eqref{3.31} is non-random. Hence, we have
\begin{equation}\label{ast_ast}
\eta(r) = \d_x^3 \tu \text{ on } \{\tu > r\} \text{ for every } r>0.
\end{equation}
Now, since $\tu \geq 0$, for any $\f \in L^\infty(\tilde{\Omega}, \tilde{F}, \tilde{P}; L^{\infty}(Q_T)), r>0 $ and $\ve \in (0,1]$ we have
\begin{eqnarray}\label{5.18}
\nonumber & & I_\ve = \tilde{\E}  \int_{Q_T} f_\ve(\tu_\ve) (\d_x^3 \tu_\ve)  \f \, dx  dt = \tilde{\E}  \int_{Q_T} f_\ve(\tu_\ve) (\d_x^3 \tu_\ve) \chi_{\{\|\tu_\ve - \tu\|_{L^{\infty}(Q_T)} < \frac{r}{2}\} \cap \{\tu > r\}} \, \f \,  dx  dt\\
 & &  + \tilde{\E}  \int_{Q_T} f_\ve(\tu_\ve) (\d_x^3 \tu_\ve) \chi_{\{\|\tu_\ve - \tu\|_{L^{\infty}(Q_T)} \geq \frac{r}{2}\} \cup \{\tu \leq  r\}} \, \f \,  dx  dt := J_1(r,\ve) + J_2(r,\ve).
\end{eqnarray}
Since $f_\ve(s)$ converges to $s^2$, as well as $f_\ve'(s)$ converges to $2s$ uniformly on any interval $[0,M]$, it follows from \eqref{3.31}, \eqref{5.17} and \eqref{ast_ast}, that
\begin{equation}\label{5.19}
J_1(r,\ve) \to \tilde{\E} \int_{Q_T} \tu^2 (\d_x^3 \tu) \chi_{\hat{u}>r} \f \, dx dt \text{ as } \ve \to 0.
\end{equation} 
As for the second term $J_2(r,\ve)$, we have
\begin{multline}\label{5.20}
|J_2(r,\ve)| \leq \|\f\|_{L^\infty(\tilde{\Omega}, \tilde{F}, \tilde{P}; L^{\infty}(Q_T))} \left(\tilde{\E} \int_{Q_T} f_\ve(\tu_\ve) (\d_x^3 \tu_\ve)^2 \, dx \, dt\right)^{\frac{1}{2}} \\
\times \left(\tilde{\E}  \int_{Q_T} f_\ve(\tu_\ve) \chi_{\{\|\tu_\ve - \tu\|_{L^{\infty}(Q_T)} \geq \frac{r}{2}\} \cup \{\tu \leq  r\}}  dx \, dt \right)^{\frac{1}{2}}.
\end{multline}
Since $f_\ve(\tu_\ve)$ converges to $\tu^2$ as $\ve \to 0$, it follows from \eqref{3.31}, that
\begin{equation}\label{5.21}
\tilde{\E}  \int_{Q_T} f_\ve(\tu_\ve) \chi_{\{\|\tu_\ve - \tu\|_{L^{\infty}(Q_T)} \geq \frac{r}{2}\} \cup \{\tu \leq  r\}}  dx \, dt \to \tilde{\E}   \int_{Q_T} \tu^2 \chi_{\{\tu \leq r\}} \, dx \, dt  = O(r^2)
\end{equation}
as $\ve \to 0$.  Thus, \eqref{5.20} and \eqref{5.21}, combined with \eqref{3.30}, yield
\begin{equation*}\label{5.22}
J_2(r,\ve) \leq C r \|\f\|_{L^\infty(\tilde{\Omega}, \tilde{F}, \tilde{P}; L^{\infty}(Q_T))},
\end{equation*}
for some $C>0$. From \eqref{5.13b}, for any $\f \in L^\infty(\tilde{\Omega}, \tilde{F}, \tilde{P}; L^{\infty}(Q_T))$ we have
\begin{equation}\label{5.23}
I_\ve \to \int_{Q_T} \eta \f \, dx \, dt:=I \text{ as } \ve \to 0, \P - \text{ a.s. },
\end{equation}
where $\eta = \eta(t,x,\w) \in L^2(Q_T)$. Using \eqref{3.30} and \eqref{3.31} once again,
\begin{multline}\label{5.24}
\tilde{\E} \left(\int_{Q_T} f_\ve(\tu_\ve) (\d_x^3 \tu_\ve) \f \, dx \, dt\right)^2 \\
\leq \|\f\|^2_{L^\infty(\tilde{\Omega}, \tilde{F}, \tilde{P}; L^{\infty}(Q_T))} \sqrt{\tilde{\E} \left(\int_{Q_T} f_\ve(\tu_\ve)  \, dx \, dt\right)^2}  \sqrt{ \tilde{\E} \left( \int_{Q_T} f_\ve(\tu_\ve) (\d_x^3 \tu_\ve)^2 \, dx \, dt\right)^2}\\
\leq C  \|\f\|^2_{L^\infty(\tilde{\Omega}, \tilde{F}, \tilde{P}; L^{\infty}(Q_T))} \sqrt{\tilde{\E} \|\d_x u_0\|_2^2} < \infty.
\end{multline}
By Vitali's Convergence Theorem, \eqref{5.23} and \eqref{5.24}  yield 
\begin{equation*}\label{5.25}
I_\ve = \tilde{\E} \int_{Q_T} f_\ve(\tu_\ve)  (\d_x^3 \tu_\ve) \f \, dx \, dt \to \tilde{\E} \int_{Q_T} \eta \f \, dx \ dt = I.
\end{equation*} 
Due to \eqref{5.18} and \eqref{5.19},
\begin{multline*}
\lim_{\ve \to 0} I_\ve = \limsup_{\ve \to 0} I_\ve  = \lim_{\ve \to 0} I_1(r,\ve) + \limsup_{\ve \to 0} I_2(r,\ve) =  \tilde{\E} \int_{Q_T} \tu^2 (\d_x^3 \tu) \chi_{\hat{u}>r} \f \, dx dt + \limsup_{\ve \to 0} I_2(r,\ve).
\end{multline*}
This way
\begin{equation}\label{5.26}
\left|I-  \tilde{\E} \int_{Q_T} \tu^2 (\d_x^3 \tu) \chi_{\hat{u}>r} \f \, dx dt \right| \leq \limsup_{\ve \to 0} I_2(r,\ve) = O(r).
\end{equation}
It remains to pass to the limit in \eqref{5.26} as $r \to 0$. Using \eqref{5.17}, \eqref{ast_ast} and the uniform convergence of $f_\ve(\tu_\ve)$ to $\tu^2$ $\P$ -almost surely, we have
\begin{equation*}\label{5.27}
\left| \tilde{\E} \int_{Q_T} \tu^2 (\d_x^3 \tu) \chi_{\hat{u}>r} \f \, dx dt \right| = \lim_{\ve \to 0} \left|  \tilde{\E}  \int_{Q_T} f_\ve(\tu_\ve) (\d_x^3 \tu_\ve) \chi_{\{\|\tu_\ve - \tu\|_{L^{\infty}(Q_T)} < \frac{r}{2}\} \cap \{\tu > r\}} \, \f \,  dx  dt  \right| 
\end{equation*}
for any $\f \in L^2(\tilde{\Omega}, \tilde{\F}, \tilde{\P}; L^2(Q_T))$. For any $\ve>0$, \eqref{3.30} and \eqref{3.31} imply
\begin{eqnarray*}\label{5.28}
\nonumber & &\left|  \tilde{\E}  \int_{Q_T} f_\ve(\tu_\ve) (\d_x^3 \tu_\ve) \chi_{\{\|\tu_\ve - \tu\|_{L^{\infty}(Q_T)} < \frac{r}{2}\} \cap \{\tu > r\}} \, \f \,  dx  dt  \right|  \leq  \tilde{\E}  \int_{Q_T} f_\ve(\tu_\ve) |\d_x^3 \tu_\ve| \, |\f| \,  dx  dt \\
 & &\leq \left( \tilde{\E}  \int_{Q_T} f_\ve(\tu_\ve) (\d_x^3 \tu_\ve)^2  \,  dx  dt\right)^{\frac{1}{2}} \left( \tilde{\E}  \int_{Q_T} f_\ve(\tu_\ve) \f^2  \,  dx  dt\right)^{\frac{1}{2}} \leq C \left( \tilde{\E}  \int_{Q_T} \f^2  \,  dx  dt\right)^{\frac{1}{2}},
\end{eqnarray*}
where $C$ is independent of $r$ and $\ve$. This way, the formula 
\[
\E  \int_{Q_T} \tu^2 (\d_x^3 \tu) \chi_{\{\tu > r\}} \f \, dx  dt 
\]
defines a linear continuous functional on $L^2(\tilde{\Omega}, \tilde{\F}, \tilde{\P}; L^2(Q_T))$, and thus, by Rietz Representation Theorem,
\begin{equation}\label{5.29}
\E  \int_{Q_T} \tu^4 (\d_x^3 \tu)^2 \chi_{\{\tu > r\}} \, dx  dt \leq C.
\end{equation}
Combined with the convergence $\tu^2 (\d_x^3 \tu) \chi_{\{\tu > r\}} \to \tu^2 (\d_x^3 \tu) \chi_{\{\tu > 0\}}$ as $r \to 0$, and Vitali's Convergence Theorem, \eqref{5.29} yields
\[
\tilde{J} = \tu^2 (\d_x^3 \tu) \chi_{\{\tu > 0\}}
\]
$\P$- almost surely, which concludes the proof of the Proposition. 
\end{proof}
Combined with \eqref{3.30}, Proposition \ref{prop5.4} yields the following Corollary:
\begin{corollary}\label{cor5.6}
We have
\[
\tu \in L^{\infty}(0,T; H^1(\T)) \cap C^{\frac{\gamma}{4}, \gamma}(Q_T)
\]
for $\gamma<\frac{1}{2}$, $\tilde{\P}$ almost surely. 
\end{corollary}
\subsection{Passing to the limit in the regularized problem}
In this subsection we are going to pass to the limit in \eqref{1.5} as $\ve \to 0$. For any $\f \in C^{\infty}(\T)$ we have
\begin{multline*}\label{5.30}
     (\ti{u}_\ve(t,\cdot), \f)_2 -  (u_{0\ve}, \f)_2 = \int_0^t  \int_{L} f_\ve(\ti{u}_\ve (\d_x^3 \ti{u}_\ve(s,\cdot)) \d_x \f dx ds  
    - \frac{1}{2}  \int_0^t ( \d_x( \ti{u}_\ve(s, \cdot)), \d_x \f)_2 ds \\
    - \int_0^t \left(\int_{L} \ti{u}_\ve \d_x \f dx\right) d \tilde{\beta}(s),
\end{multline*}
for all $t \in [0,T]$. In the same way as in \cite{Gess2}, we have, for all $t \in [0,T]$, 
\begin{equation}\label{5.31}
(\tu_\ve(t), \f)_2 \to (\tu(t), \f)_2,
\end{equation} 
\begin{equation}\label{5.32}
 \int_0^t  \int_{L} f_\ve(\ti{u}_\ve) (\d_x^3 \ti{u}_\ve(s)) \d_x \f dx ds \to \int_0^t  \int_{\{\tu(s)>0\}} \tu^2 (\d_x^3 \ti{u}(s)) \d_x \f dx ds
\end{equation}
and
\begin{equation}\label{5.33}
\int_0^t (\d_x \tu_\ve, \d_x \f)_2 \, ds \to \int_0^t (\d_x \tu, \d_x \f)_2 \, ds
\end{equation}
as $\ve \to 0$, $\P$-almost surely. 
From \eqref{5.13a} we get $\|\tu_\ve - \tu\|_{C(Q_T)} \to 0$ as $\ve \to 0$ $\P$-almost surely, and \eqref{5.31} follows. The convergence \eqref{5.32} follows from the weak convergence 
\[
\tilde{J}_\ve = f_\ve(\tu_\ve) \d_x^3 \tu_\ve \rightharpoonup \tilde{J} = \chi_{\{\tu>0\}} \tu^2 (\d_x^3 \tu)
\]
in $L^2(Q_T)$  (see \eqref{5.13b} and Proposition \ref{5.4}).
In order to establish \eqref{5.33}, note that
\begin{multline*}
\int_0^t (\d_x \tu_\ve, \d_x \f)_2 \, ds = -\int_0^t \int_L \tu_\ve (\d_x^2 \f) \, dx \ ds \to  -\int_0^t \int_L \tu (\d_x^2 \f) \, dx \ ds = \int_0^t \int_L \d_x \tu \d_x \f \, dx \ ds,
\end{multline*}
since $\tu \in H^1(\T)$. Thus
\begin{multline*}
M_{\ve, \f}(t) = - \int_0^t (\tu_\ve(s), \d_x \f)_2 \, d \tilde{\beta}_\ve(s) = (\tu_\ve, \f)_2 - (\tu_\ve(0), \f)_2 - \int_0^t (f_\ve(\tu_\ve)(\d_x^3 \tu_\ve), \d_x \f)_2 \, ds \\
+ \frac{1}{2} \int_0^t (\d_x \tu_\ve, \d_x \f)_2 \, ds \to (\tu, \f)_2 - (\tu(0), \f)_2 - \int_0^t \int_{\{\tu(s)>0\}} \tu^2 (\d_x^3 \tu), \d_x \f \, dx ds \\
+  \frac{1}{2} \int_0^t (\d_x \tu_\ve, \d_x \f)_2 \, ds, \ \ \ve \to 0.
\end{multline*}
In this case, it is sufficient to show that for all $t \in [0,T]$
\[
M_\f(t) = - \int_0^t (\tu(s), \d_x \f)_2 \, d \tilde{\beta}(s). 
\]
The proof of this fact follows the lines of \cite{Gess2}, p.229, using \cite{Hofm}, Proposition A.1. Hence, we established the existence of a martingale solution of \eqref{1.1} on any interval $[0,T], T>0$. In the same way as in the proof of Theorem \ref{Thm4.1}, this solution may be extended globally. In order to complete the proof of the main result, we need to pass to the limit as $\ve \to 0$ in the inequality \eqref{3.35}. From Corollary \ref{cor5.6} we get that $u$ is continuous and bounded $\P$-almost surely on $Q_T$. Since $u \in L^\infty(0,T; H^1(\T))$, for fixed $t_0 \in [0,T]$ we can find $t_n(\omega) \to t_0$ such that $\tu(t_n, \cdot) \in H^1(\T)$ and $\|\tu(t_n, \cdot)\|_{H^1(\T)}$ are uniformly bounded in $n$ $\P$-almost surely. For any $\f \in C^\infty(\T)$, using the continuity of $\tu$ and  Lebesgue Dominated Convergence Theorem, we have
\begin{equation}\label{5.34}
\int_L \tu(t_n, x) \d_x \f \, dx \to \int_L \tu(t_0, x) \d_x \f \, dx, \ n \to \infty.
\end{equation}
On the other hand, since $\{\tu_x (t_n, \cdot)\}_n$ is bounded in $L^2(\T)$ $\P$-almost surely, for every $\omega \in \tilde{\Omega}$ we have the existence of a subsequence, along which 
\begin{equation}\label{5.35}
\int_L \d_x \tu(t_n, x) \f(x) \, dx \to \int_L \eta \f(x) \, dx, \ n \to \infty
\end{equation}
$\P$-almost surely for  some $\eta \in L^2(Q_T).$ Taking 
\[
\int_L \tu(t_n, x) \d_x \f \, dx = - \int_L \d_x \tu(t_n, x)  \f \, dx
\]
into account, \eqref{5.34} and \eqref{5.35} yield 
\[
\int_{L} \tu(t_0, x) \d_x \f dx = -\int_L \eta \f(x) \, dx.
\] 
Thus $\d_x \tu(t_0,x) = \eta \in L^2(Q_T)$ $\P$-almost surely. The measurability of $\d_x \tu(t_0,x,\w)$ follows from the Proposition \ref{prop5.4}. Hence $\tu(t,\cdot) \in H^1(\T)$ for all $t \in [0,T]$. The convergence of $\tu_\ve$ to $\tu$ also implies 
\[
\int_L \tu_\ve(t,x) (\d_x \f) \, dx \to \int_L \tu(t,x) (\d_x \f) \, dx \text{ for all } t \in [0,T]
\] 
for any $\f \in C^\infty(\T)$ $\tilde{\P}$- almost surely. The latter is equivalent to 
\[
\d_x \tu_\ve(t,\cdot) \rightharpoonup \d_x \tu(t,\cdot) \text{ in } L^2(\T) \text{ as } \ve \to 0
\] 
$\tilde{\P}$- almost surely, hence
\[
\int_L (\d_x \tu(t,x))^2 \, dx \leq \liminf_{\ve \to 0} \int_L (\d_x \tu_\ve(t,x))^2 \, dx. 
\]
Thus, using \eqref{3.32} and passing to the limit in \eqref{3.35} as $\ve \to 0$, and taking into account \eqref{astast}, the inequality \eqref{2.2} follows, which completes the proof.
\end{proof}

\section*{Acknowledgments}  The research of Oleksandr Misiats was supported by Simons Collaboration
Grant for Mathematicians No. 854856. The work of Oleksandr Stanzhytskyi and Oleksiy Kapustyan was partially supported by the Ukrainian Government Scientific Research Grant No. 210BF38-01.

\bibliographystyle{plain}

\end{document}